\documentclass%[12pt]
{amsproc}
\usepackage{amsmath}
\usepackage{amsthm}
\usepackage{amssymb}
\usepackage{graphicx}
\usepackage{multirow}

\DeclareGraphicsExtensions{.png,.pdf,.eps}
\usepackage[all,2cell]{xy}
\usepackage{tikz}
\usepackage{pgf}
\usepackage{enumerate}
\usepackage{mathdots}
\usetikzlibrary{cd}
\usepackage{caption}
\usepackage{array}
\newcolumntype{H}{>{\setbox0=\hbox\bgroup}c<{\egroup}@{}}
\usepackage{hyperref}

\newtheorem{theorem}{Theorem}[section]
\newtheorem{lemma}[theorem]{Lemma}
\newtheorem{corollary}[theorem]{Corollary}
\newtheorem{proposition}[theorem]{Proposition}

\theoremstyle{definition}

\newtheorem{definition}[theorem]{Definition}
\newtheorem{notation}[theorem]{Notation}
\newtheorem{setup}[theorem]{Setup}

\newtheorem{remark}[theorem]{Remark}

\newtheorem{set}[theorem]{Setup}

\newtheorem{construction}{Construction}

\title[Mori dream spaces arising from $\mathbb{C}^*$-actions]{Small modifications of Mori dream spaces\\arising from $\mathbb{C}^*$-actions}

\author[Occhetta]{Gianluca Occhetta}
\address{Dipartimento di Matematica, Universit\`a degli Studi di Trento, via
Sommarive 14 I-38123 Povo di Trento (TN), Italy}

\email{gianluca.occhetta@unitn.it, eduardo.solaconde@unitn.it}

\author[Romano]{Eleonora A. Romano}
\address{Dipartimento di Matematica, Universit\`a degli Studi di Genova, via Dodecaneso 35, I-16146, Genova (GE), Italy}
\email{eleonoraanna.romano@unige.it}

\author[Sol\'a Conde]{Luis E. Sol\'a Conde}

\author[Wi\'sniewski]{Jaros\l{}aw A. Wi\'sniewski}
\address{Instytut Matematyki UW, Banacha 2, 02-097 Warszawa, Poland}
\email{J.Wisniewski@uw.edu.pl}


\thanks{First and third author supported by PRIN project ``Geometria delle variet\`a algebriche''. Second and fourth author supported by Polish National Science Center project 2016/23/G/ST1/04282. Second author supported by Associazione Amici di Claudio Dematt\`e.}

\usepackage[textsize=tiny]{todonotes}

\usepackage{enumitem}
\newcommand\ignore[1]{}

\DeclareMathOperator{\HH}{H}

\newcommand\CC{{\mathbb{C}}}
\newcommand\PP{{\mathbb{P}}}

\newcommand\QQ{{\mathbb{Q}}}

\def\B{{\mathbb B}}
\def\C{{\mathbb C}}

\def\N{{\mathbb N}}
\def\P{{\mathbb P}}
\def\Q{{\mathbb Q}}
\def\R{{\mathbb R}}
\def\Z{{\mathbb Z}}

\def\cB{{\mathcal B}}

\def\cD{{\mathcal D}}

\def\cL{{\mathcal L}}
\def\cM{{\mathcal M}}
\def\cN{{\mathcal{N}}}
\def\cO{{\mathcal{O}}}

\def\cS{{\mathcal S}}

\def\cY{{\mathcal Y}}
\def\Q{{\mathbb{Q}}}

\def\fg{{\mathfrak g}}
\def\fh{{\mathfrak h}}

\def\operatorname#1{\mathop{\rm #1}\nolimits}

\def\DA{{\rm A}}
\def\DB{{\rm B}}
\def\DC{{\rm C}}
\def\DD{{\rm D}}
\def\DE{{\rm E}}
\def\DF{{\rm F}}
\def\DG{{\rm G}}

\def\Proj{\operatorname{Proj}}

\def\Hom{\operatorname{Hom}}

\def\Pic{\operatorname{Pic}}
\def\Hom{\operatorname{Hom}}

\def\id{\operatorname{id}}

\def\rk{\operatorname{rk}}

\def\Bs{\operatorname{Bs}}
\def\Bss{\mathcal{B}}

\def\gen{\operatorname{gen}}

\def\Nef{{\operatorname{Nef}}}

\def\NU{{\operatorname{N^1}}}

\def\Mov{{\operatorname{Mov}}}
\newcommand{\cMov}[1]{\overline{\Mov(#1)}}

\def\ad{\operatorname{ad}}

\def\SX{\mathcal{S}\!X}
\def\GX{\mathcal{G}\!X}

\renewcommand{\ss}{\operatorname{ss}}

\newcommand{\pb}{\ar@{}[dr]|{\text{\pigpenfont J}}}

\newcommand{\shse}[3]{0 ~\ra ~#1~ \lra ~#2~ \lra ~#3~ \ra~ 0}
\makeatletter
\newcommand{\xleftrightarrow}[2][]{\ext@arrow 3359\leftrightarrowfill@{#1}{#2}}
\makeatother
\newcommand{\xdasharrow}[2][->]{
\tikz[baseline=-\the\dimexpr\fontdimen22\textfont2\relax]{
\node[anchor=south,font=\scriptsize, inner ysep=1.5pt,outer xsep=2.2pt](x){#2};
\draw[shorten <=3.4pt,shorten >=3.4pt,dashed,#1](x.south west)--(x.south east);
}}

\newcommand\m{{\mathfrak m}}

\newcommand\ra{{\ \rightarrow\ }}
\newcommand\lra{\longrightarrow}

\def\Mo{\operatorname{\hspace{0cm}M}}

\begin{document}
\begin{abstract}
We link small modifications of projective varieties with a $\CC^*$-action to their GIT quotients. Namely, using flips with centers in closures of Bia{\l}ynicki-Birula cells, we produce a system of birational equivariant modifications of the original variety, which includes those on which a quotient map extends from a set of semistable points to a regular morphism. 

The structure of the modifications is completely described for the blowup along the sink and the source of smooth varieties with Picard number one with a $\CC^*$-action which has no finite isotropy for any point. Examples can be constructed upon   homogeneous varieties with a $\CC^*$-action associated to short grading of their Lie algebras.
\end{abstract}
\maketitle
\tableofcontents

%!TEX root = WORS3.tex

\section{Introduction}\label{sec:intro}

In classical Mumford's Geometric Invariant Theory (GIT), \cite{MFK}, quotients of projective varieties by reductive group actions are determined by the choice of a linearization of an ample line bundle which yields the set of semistable points, an open subset of the original variety on which the quotient is defined as a morphism; thus GIT requires passing to quasiprojective varieties. 

However, if the group in question is $\CC^*$ the complement of the set of semistable points is a union of closures of Bia{\l}ynicki-Birula cells, whose structure is well understood, see \cite{BBS}.
One may then ask if there is a natural equivariant compactification of the set of semistable points on which the quotient map extends to a regular morphism. This requires a good understanding of the equivariant birational modifications of varieties with a $\CC^*$-action.

The relation of birational geometry with the theory of quotients of $\CC^*$-actions has been acknowledged since the early days of the Minimal Model Program or Mori Theory; see the contributions by Thaddeus and Reid \cite{Thaddeus1996}, \cite{ReidFlip} (see also \cite[Remark 11.1.2, p.~46]{BBC}) which focused on describing  birational transformations in terms of variation of stability conditions yielding geometric quotients. This concept gave rise to the notions of Mori Dream Space (MDS) introduced in \cite{HuKeel}, and total coordinate -- or Cox -- ring (see \cite{ADHL}), whose spectrum gives, as GIT quotients, small $\QQ$-factorial modifications of an MDS.

W{\l}odarczyk in  \cite{Wlodarczyk} used $\CC^*$-actions to prove the Weak Factorization Conjecture, which asserts that a birational map of smooth projective varieties can be factored as a sequence of blowups and blowdowns in smooth centers. The key tool in his work was the notion of birational cobordism,  constructed by Morelli in the toric case \cite{Morelli}. This is a quasiprojective variety with a $\CC^*$-action, containing two open equivariant subsets admitting quotients to two birationally equivalent varieties; see  \cite{AKMW} and \cite{Wlodarczyk2009} for a broader view on this topic.

In the present paper we deal with equivariant birational modifications of smooth projective varieties with a $\CC^*$-action. When the action is equalized, which means that no point has finite isotropy, after blowing up the source and the sink of the action (see Section \ref{ssec:actions} for definitions) we obtain a variety which admits a system of small $\QQ$-factorial $\C^*$-equivariant modifications (cf. Section  \ref{sec:movbordism}).
Some of these modifications are compactifications of sets of stable points which are $\PP^1$-bundles over the respective quotients (Proposition \ref{prop:P1model}). Each of these modifications is a projective version of a cobordism associated to the natural birational map between a pair of GIT quotients. Namely, each of the varieties that we construct is a %{\em bordisms}; these are 
$\C^*$-equivariant compactification of a cobordism. Its boundary consists of two divisors -- that are the source and the sink of the action -- which are birationally equivalent varieties in W{\l}odarczyk's construction of the cobordism.  Roughly speaking, cobordisms represent birational maps, and we study here their birational modifications.

More precisely, the main results of the paper can be summarized in the following statement (see Sections \ref{ssec:actions}, \ref{ssec:BBS} for notation and definitions of the concepts involved). For clarity of exposition, we state here only the case in which the extremal fixed point components are positive dimensional.  Our arguments work also in the remaining cases; a complete description can be found in Section \ref{sec:movbordism}.

\begin{theorem}\label{thm:main}
Let $X$ be a smooth complex projective variety of Picard number one admitting a nontrivial $\C^*$-action. Assume that the action is equalized of criticality $r$, and that its extremal fixed point components $Y_0,Y_r$ are not isolated points. Denote by $\GX(i,i+1)$, $i=0,\dots,r-1$, the corresponding geometric quotients. Then:
\begin{enumerate}
\item The varieties $\GX(i,i+1)$ are smooth and the natural birational maps
$$\xymatrix{
\GX(0,1)\ar@{-->}[r]& \GX(1,2)\ar@{-->}[r]& \ldots \ar@{-->}[r]& %\GX(r-2,r-1)\ar@{-->}[r]&
 \GX(r-1,r)}
$$
are flips. %except in the following two cases:
%\begin{itemize}
%\item if $Y_0$ is a point then $\GX(0,1)\dashrightarrow \GX(1,2)$ is a blowdown;
%\item if $Y_r$ is a point then $\GX(r-2,r-1)\dashrightarrow \GX(r-1,r)$ is a blowup.
%\end{itemize}
\item The blowup $X^\flat$ of $X$ along $Y_0$, $Y_r$ is a Mori dream space.
\item Given a pair $(i,j)$ of indices $i,j\in\{0,\dots,r\}$, $i\leq j$, %, different from:
%\begin{itemize}
%\item $(0,1)$ if $Y_0$ is a point, and from
%\item $(r-1,r)$ if $Y_r$ is a point,
%\end{itemize}
there exists a unique small $\Q$-factorial modification $X(i,j)$ of $X^\flat$ that is smooth and admits a $\C^*$-action with extremal fixed point components $\GX(i,i+1)$, $\GX(j,j+1)$.
\item Every small $\Q$-factorial modification of $X^\flat$ is constructed as above.
\end{enumerate}
\end{theorem}

Note that in the above statement we are using the word {\em flip} to refer to a $D$-flip, for a certain  $\Q$-divisor $D\in\Pic(X^\flat)\otimes_{\Z}\Q$ (see \cite[Definition~1.9]{HK}). More concretely, we will see that the flips appearing in our description are compositions of a smooth blowup and a smooth blowdown (Section \ref{sec:bord}). In order to prove the previous statement we will describe completely the small modifications among the varieties $X(i,j)$, showing that they can be written as a composition of $\C^*$-equivariant flips, whose restriction to the extremal fixed point components are the flips described in Part (1) of Theorem (Section \ref{sec:movbordism}).

\medskip

\noindent{\bf Outline:} We start the paper with a section of preliminaries on $\C^*$-actions and their GIT quotients; we also  include some results on $\C^*$-invariant linear systems on the variety $X$ and its blowup $X^{\flat}$ along its extremal fixed point components. Section \ref{sec:bord} contains the technical core of the paper: we construct some $\C^*$-equivariant flips of varieties admitting an equalized $\C^*$-action with codimension one sink and source, relating them to flips of these extremal fixed point components (Theorem \ref{thm:flip}). By applying recursively Theorem \ref{thm:flip}, we prove Theorem \ref{thm:main} in Section \ref{sec:movbordism}, describing completely the movable cone of the variety $X^\flat$ of the statement and its Mori chamber decomposition. Finally we illustrate our results by showing how to construct examples of equalized $\C^*$-actions in the framework of representation theory and rational homogeneous spaces (Section \ref{sec:RH}). In particular we revisit some examples of equalized $\C^*$-actions of small bandwidth, previously presented in \cite{WORS1}.  

\medskip

\noindent{\bf Acknowledgements.} The authors would like to thank Joachim Jelisiejew for his valuable help with questions regarding GIT. The third author would like to thank Fran Presas for the inspiring conversations about Morse theory 20 years ago.

%!TEX root = WORS3.tex

\section{Preliminaries}\label{sec:prelim}

\subsection{Actions of complex tori}\label{ssec:actions}

We will introduce here some notation and conventions on $\C^*$-actions. We refer to \cite{WORS1,CARRELL,BB} for further details. 

\subsubsection*{Fixed point components}\label{sssec:fxd} Given a $\C^*$-action on a proper complex algebraic variety $X$, we denote by $X^{\C^*}$  the fixed locus of the action, and by $\cY$ the set of irreducible fixed point components of the action, so that we may write: 
$$X^{\C^*}=\bigsqcup_{Y\in \cY}Y.$$ 
Among these components there are two distinguished ones, called {\em sink} and {\em source} of the action, defined by the property of containing, respectively, the limiting points:
$$
\lim_{t\to 0}t^{-1}x, \quad \lim_{t\to 0}tx
$$
where $x\in X$ is a general point. If $X$ is smooth and projective, which will be mostly our case, every $Y\in\cY$ is smooth (cf. \cite{IVERSEN}).

\subsubsection*{Linearizations and weight maps}\label{sssec:lin} Given a $\C^*$-action on a normal projective variety $X$ as above, and given a line bundle $L\in \Pic(X)$, one may find a linearization of the $\C^*$-action on it (cf. \cite{KKLV}), so that for every $Y\in \cY$, $\C^*$ acts on $L_{|Y}$ by multiplication with a character $m\in \Mo(\C^*)=\Hom(\C^*,\C^*)$, that we call {\em weight of the linearization on $Y$}. It is well known that any two linearizations differ by a character of $\C^*$, so that for every line bundle $L$ there exists a unique linearization (called {\em normalized})  whose weight at the sink is equal to zero. By fixing an isomorphism $\Mo(\C^*)\simeq \Z$, this linearization defines a map $\mu_L:\cY\to\Z$, sending every fixed point component to its weight.

Abusing notation we will denote with the same symbols line bundles and the Cartier divisors defining them, and use the additive notation for the group operation in $\Pic(X)$. Note that
$$
\mu_{kL}(Y)=k\mu_L(Y),\quad \mu_{L+L'}(Y)=\mu_{L}(Y)+\mu_{L'}(Y),
$$
for every $L,L'\in\Pic(X)$, $k\in \Z$, $Y\in \cY$.
In particular we may extend this definition to $\Q$-divisors in $X$, by setting:
$$
\mu_{qL}(Y):=q\mu_L(Y)\in\Q, \quad\mbox{for }q\in\Q,\,\, L\in\Pic(X),\,\, Y\in \cY.
$$

\subsubsection*{Actions on polarized pairs}\label{sssec:polar} A $\C^*$-action on a projective variety $X$, endowed with the weight map $\mu_L$ determined by an ample $\Q$-divisor $L$, will be referred to as a {\em $\C^*$-action on the $\Q$-polarized pair $(X,L)$}. In this case one may denote by 
$$
a_0<\dots<a_r,$$ 
the weights $\mu_L(Y)$, $Y\in\cY$, ordered increasingly, and set: 
$$
Y_i:=\bigcup_{\mu_L(Y)=a_i}Y.
$$ 
In analogy with the case of Morse theory, the values $a_i$ will be called the {\em critical values}; the number $r$ will be called the {\em criticality of the $\C^*$-action.}
It is well known that the minimum and maximum of these values are achieved at the sink and the source of the action, respectively, so that, $Y_0,Y_r$ are respectively the sink and the source of the action and, in particular, $a_0=0$. The value $\delta=a_r$ is called the {\em bandwidth} of the $\C^*$-action on $(X,L)$. 

\subsubsection*{The Bia{\l}ynicki-Birula decomposition}\label{sssec:BBcells}
Let $X$ be a proper variety admitting a $\C^*$-action as above. 
Given $Y\in \cY$, we denote by 
\begin{equation}\label{eq:BBcells}
\begin{array}{l}X^\pm(Y):=\{x\in X|\,\, \lim_{t^{\pm 1}\to 0} tx\in Y\},\\[3pt]
B^\pm(Y):=\overline{X^\pm(Y)},
\end{array}
\end{equation}
the {\it Bia{\l}ynicki-Birula cells} of the action and their closures; we refer to \cite{CARRELL} for a complete account on the Bia{\l}ynicki-Birula decomposition and its applications, and to \cite{BB} for the original reference. 
When considering the action of $\C^*$ on a $\Q$-polarized pair $(X,L)$ as above, we will write:
$$B^\pm_i:=\bigcup_{\mu_L(Y)=a_i}B^\pm(Y) $$
and, for notational reasons, we also set $B^\pm_i=\emptyset$ for $i\in \Z\setminus\{0,\dots,r\}$.

If $X$ is smooth, the normal bundle of $Y$ in $X$ splits into two subbundles, on which $\C^*$ acts with positive and negative weights, respectively:
\begin{equation}
\cN_{Y|X}\simeq \cN^+(Y)\oplus \cN^-(Y).
\label{eq:normal+-}
\end{equation}
We will often use the notation:
\begin{equation}\label{eq:Vpm}
\begin{array}{l}
\nu^{\pm}(Y):=\rk \cN^{\pm}(Y),\\[3pt]
V^\pm(Y):=\P(\cN^{\pm}(Y)^\vee), \quad \displaystyle V^\pm_i:=\bigcup_{\mu_L(Y)=a_i}V^\pm(Y).
\end{array}
\end{equation}
The action of $\C^*$ on $X^{\pm}(Y)$ is equivariantly isomorphic to the induced action on the bundles $\cN^{\pm}(Y)$ (see \cite[Theorem~4.2]{CARRELL} and \cite{BB} for the original exposition).

\subsubsection*{Equalized actions}\label{sssec:equal}

We say that the action of $\C^*$ on a proper variety $X$ is {\em  equalized} at $Y\in\cY$ if for every $x\in \big(X^{-}(Y)\cup X^{+}(Y)\big)\setminus Y$ the isotropy group of the $\C^*$-action on $x$ is trivial. Note that in the smooth case, the definition of equalization presented here coincides with the one introduced in \cite[Definition~1.6]{RW}: 
\begin{lemma} \label{lem:equalized}
Let $X$ be a smooth variety supporting a $\C^*$-action. The action is equalized if and only if the weights of the action on $\cN^{\pm}(Y)$ are all equal to $\pm 1$ for every fixed point component $Y\in\cY$.
\end{lemma}
\begin{proof}
It follows from the existence, for every $Y\in \cY$, of the $\C^*$-equivariant isomorphisms $X^{\pm}(Y)\simeq \cN^{\pm}(Y)$. 
\end{proof}

Moreover, the equalization hypothesis implies that the closure of any $1$-dim\-ens\-ional orbit is a smooth rational curve, whose $L$-degree may be computed in terms of the weights at its extremal points. The following statement follows from the AM vs. FM formula presented in \cite{RW}: 

\begin{lemma}[AM vs. FM]\label{lem:AMFM}
Let $(X,L)$ be a $\Q$-polarized pair, with $X$ smooth, supporting an equalized action of $\C^*$, and let $C$ be the closure of a $1$-dimensional orbit, whose sink and source are denoted by $x_-$ and $x_+$. Then $C$ is a smooth rational curve of $L$-degree  equal to $\mu_L(x_+)-\mu_L(x_-)$.
\end{lemma} 
\begin{proof}
The smoothness follows from the Bia{\l}ynicki-Birula decomposition (see \cite[Theorem~4.2]{CARRELL}), while the degree statement is precisely \cite[Corollary 3.2 (c)]{RW}.
\end{proof}

\subsubsection*{B-type actions and bordisms}\label{sssec:btype}
Following \cite{WORS1}, a $\C^*$-action on a proper algebraic variety $X$ whose extremal fixed point components $Y_0,Y_r$ have codimension one is called a {\em B-type action}. In the case in which $X$ is projective and smooth, by the Bia{\l}ynicki-Birula decomposition, the restriction maps $\Pic(X)\to \Pic(Y_j)$, $j=0,r$ are surjective; we will say that the $\C^*$-action is a {\em bordism} (cf. \cite[Definition~3.8]{WORS1}) if the restriction maps $\imath_j^*:\Pic(X)\to \Pic(Y_j)$, $j=0,r$, fit into two short exact sequences:
$$
\begin{array}{c}
\shse{\Z[Y_r]}{\Pic(X)}{\Pic(Y_0)}
\\[2pt]
\shse{\Z[Y_0]}{\Pic(X)}{\Pic(Y_r)}
\end{array}
$$

The following equivalence, which is a consequence of \cite[Theorem~3]{CS79} has been proved in \cite[Corollary~3.7]{WORS1}.
\begin{lemma}\label{lem:bordism} Let $X$ be a smooth projective variety admitting a $\C^*$-action of B-type. The action of $\C^*$ on $X$ is a bordism if and only if $\nu^\pm(Y)\geq 2$ for every inner fixed point component $Y$. 
\end{lemma}

\subsubsection*{The blowup of a $\C^*$-variety along the extremal fixed point components}\label{sssec:blowup}
Examples of B-type $\C^*$-actions can be constructed by considering any $\C^*$-action on a proper variety smooth at its extremal fixed point component $Y_0,Y_r$. If the action is equalized at $Y_0,Y_r$, then the blowup $X^\flat$ of $X$ along $Y_0,Y_r$ inherits a $\C^*$-action that is of B-type (see \cite[Lemma~3.10]{WORS1}). As we will see later on, there is a deep relation among the birational geometry of $X^{\flat}$ and the birational geometry of the GIT quotients of $X$. In order to describe precisely this relation we will need to use study certain linear systems on $X^{\flat}$. Let us first introduce some notation.

\begin{notation}\label{not:linebundles1}
We denote by $\beta:X^\flat\to X$ the blowup of $X$ along $Y_0,Y_r$, and by $Y_i^{\flat}:=\beta^{-1}(Y_i)$, $i=0,r$ its exceptional divisors. For every $\tau_-,\tau_+\in\R$ satisfying that $0\leq \tau_-\leq\tau_+\leq \delta$ we introduce the following $\R$-divisor:
$$
L(\tau_-,\tau_+):=\beta^*L-\tau_-Y_0^\flat-(\delta-\tau_+)Y_r^\flat\in \NU(X^{\flat}).
$$
\end{notation}
We may now state the following Lemma, that will be very important later on:

\begin{lemma}\label{lem:projections2}
Let $(X,L)$ be a $\Q$-polarized pair, with $X$ smooth, admitting a $\C^*$-action, equalized at the sink and the source. For every $m\in \N,\tau_-,\tau_+\in \Q$ such that $mL\in\Pic(X)$, $0\leq \tau_-\leq \tau_+\leq \delta$, $m\tau_\pm\in\Z$, the map $\beta^*$ induces an isomorphism:
$$
\bigoplus_{k=m\tau_-}^{m\tau_+}\HH^0(X,mL)_k\simeq \HH^0(X^\flat,mL(\tau_-,\tau_+)),
$$
where $\HH^0(X,mL)_k$ is the eigenspace of $\HH^0(X,mL)$ on which $\C^*$ acts with weight $k$.
\end{lemma}

\begin{proof}
Note that $\beta^*$ induces a $\C^*$-equivariant embedding of $\HH^0(X^\flat,m(\beta^*L-\tau_-Y_0^\flat-(\delta-\tau_+)Y_r^\flat))$ into $\HH^0(X,mL)$. In order to determine its image we use \cite[Lemma~2.17]{BWW} (that describes locally the invariant sections of a line bundle along a fixed point component), that allows us to identify $\HH^0(X,mL)_k$ with the set of sections of $mL$ vanishing with multiplicity $k$ at $Y_0$, and with multiplicity $m\delta-k$ at $Y_r$, that is with $\HH^0(X,L^{\otimes m}\otimes I_{Y_0}^k\otimes I_{Y_r}^{m\delta-k})$. Then the statement follows by pulling back this identification to $X^\flat$ via $\beta$.
\end{proof}

The importance of this Lemma relies in the fact that it allows us to describe the base loci of 
divisors in $X^\flat$  
in terms of the closures of the Bia{\l}ynicki-Birula cells defined in (\ref{eq:BBcells}). Let us start by introducing the following:

\begin{notation}\label{not:linebundles2}
In the setup of Lemma \ref{lem:projections2}, let $0=a_0<\dots<a_r$ be the critical values of the action. For every $i,j\in\{0,\dots,r\}$ (see Figure \ref{fig:BSL}) we set:
$$
\begin{array}{l}
\Bss_i^+:=\{x\in X\,|\,\,\mu_L\big(\displaystyle\lim_{t\to 0} tx\big)\leq a_i\}=\bigcup_{k\leq i}B^{+}_k,\\[2pt]
\Bss_j^-:=\{x\in X\,|\,\,\mu_L\big(\displaystyle\lim_{t\to \infty} tx\big)\geq a_j\}=\bigcup_{k\geq j}B^{-}_k,\\[2pt]
\Bss_{ij}:=\Bss_i^+\cup\Bss_j^-.
\end{array}
$$

\begin{figure}[h!]
\includegraphics[width=11.5cm]{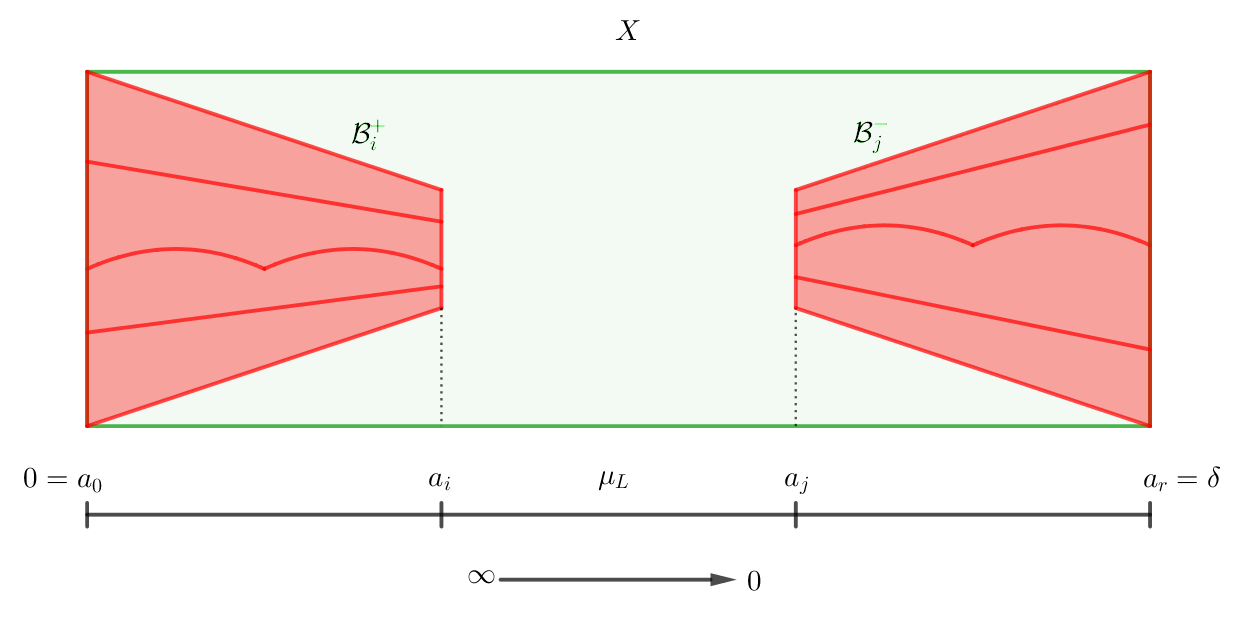}
\caption{The closed subsets $\Bss_i^+$, $\Bss_j^-$. \label{fig:BSL}}
\end{figure}

We denote by $\Bss_{ij}^\flat\subset X^\flat$ the closure of $\beta^{-1}(\Bss_{ij}\setminus (Y_0\cup Y_r))$. For  every $\tau\in[0,\delta]\cap \Q$, we set:
\begin{equation}\label{eq:ij}
i(\tau):=\min\{i \,|\,\,a_i\geq\tau\}-1,\quad j(\tau):=\max\{j\,|\,\,a_j\leq\tau\}+1.
\end{equation}
This notation is set so that, given rational numbers $\tau_-\leq \tau_+$ in $[0,\delta]$, we have 
$$
\cB_{i(\tau_-)j(\tau_+)}=\bigcup_{a_i\leq \tau_-}B_i^+
\cup \bigcup_{a_j\geq\tau_+}B_j^-.$$
\end{notation}
In particular we can state:

\begin{corollary}\label{cor:BSL}
Let $(X,L)$ be as in Lemma \ref{lem:projections2}. For every $\tau_-,\tau_+\in[0,\delta]\cap \Q$, $\tau_-\leq\tau_+$,  and for  $m\gg 0$ such that $m\tau_-,m\tau_+\in\Z$, the following  hold: 
$$
\begin{array}{l}
\Bs\left(\bigoplus_{k=m\tau_-}^{m\tau_+}\HH^0(X,mL)_k\right)=\Bss_{i(\tau_-)j(\tau_+)},\\[4pt]
\Bs\left( mL(\tau_-,\tau_+)\right)=\Bss^\flat_{i(\tau_-)j(\tau_+)}. 
\end{array}
$$
\end{corollary}

\begin{proof}
The first equality is obtained from the fact that, by definition, $$\Bs\left(\bigoplus_{k=\tau_-}^{\tau_+}\HH^0(X,mL)_k\right)=X\cap \mathbb{P}\left(\bigoplus_{\substack{k<\tau_- \text{ or }  k>\tau_+}}\HH^0(X,mL)_k\right). $$ 
 
For the second equality, in view of Lemma \ref{lem:projections2}, we are left to check that any point  $y \in (Y^\flat_0\cup Y^\flat_r)\setminus \Bss^\flat_{i(\tau_-)j(\tau_+)}$ does not belong to $\Bs\left(mL(\tau_-,\tau_+)\right)$.\\ In order to do so, we assume, for instance, that $y\in Y^\flat_0\setminus \Bss^\flat_{i(\tau_-)j(\tau_+)}$ and let $C\subset X^\flat$ be the closure of the unique $1$-dimensional $\C^*$-orbit having $y$ as sink. Denoting by $C'$ the strict transform of $C$ in $X$, and being $m\gg 0$, we may assume that $\HH^1(X,mL\otimes I_{C'})=0$, so that the $\C^*$-equivariant morphism
$$
\HH^0(X,mL)\to \HH^0(X,mL_{|C'})
$$ 
is surjective. In particular the surjectivity is inherited by the corresponding $\C^*$-eigenspaces, so that we have a surjective map:
$$
\bigoplus_{k=m\tau_-}^{m\tau_+}\HH^0(X,mL)_k\to \bigoplus_{k=m\tau_-}^{m\tau_+}\HH^0(C',mL_{|C'})_k.
$$
Applying now Lemma \ref{lem:projections2} we have a surjective map:
$$
\HH^0(X^\flat,mL(\tau_-,\tau_+))\to \HH^0(C,mL(\tau_-,\tau_+)_{|C}).
$$
By the choice of $y$, $mL(\tau_-,\tau_+)$ has positive degree on $C$, and the result follows.
\end{proof}

%%%%%%%%%%%%%%%%%%%%%%%%%%%%%%%%

\subsection{GIT-quotients of torus actions}\label{ssec:BBS}

Following  Mumford's Geometric Invariant Theory (GIT), given a reductive group $G$ acting on a variety $X$, one may consider the problem of describing all the possible proper geometric and semi-geometric quotients of $G$-invariant open subsets of $X$. In the case in which $X$ is normal and proper, and $G=\C^*$  the problem was treated in \cite{BBS1}, and the solution was written in terms of the ordered set of fixed point components of $X$. 

Although we will not need the results of \cite{BBS1} in full generality, we recall the description introduced there, since it provides a very clear geometric insight on the construction of the quotients we will work with. 

Let $X$ be a proper normal complex algebraic variety  with a nontrivial action of $\C^*$. As shown in \cite{BBS1}, there exists a unique partial order $\preceq$ on $\cY$ satisfying that $Y\preceq Y'$ if $X^+(Y')\cap X^-(Y)\neq \emptyset$, that is, if there exists an orbit converging to a point of $Y$ when $t$ goes to $\infty$, and to a point of $Y'$ when $t$ goes to $0$. Note that we have deliberately inverted the order appearing in \cite{BBS1}, in order to make it compatible with the weights of the fixed point components with respect to positive line bundles: in fact, if $L$ is a nef $\Q$-divisor in $X$, the AM vs. FM formula (\cite[Corollary~2.3]{RW}) $Y\preceq Y'$ implies that $\mu_L(Y)\leq \mu_L(Y')$.

\begin{definition}\label{def:semisection}
A {\it semi-section} of the action is a partition $\cY=\cY_-\sqcup \cY_0\sqcup\cY_+$ satisfying that:
$$
\mbox{if }\,\,Y\in \cY_-\cup \cY_0\,\,\mbox{ and }\,\, Y'\preceq Y\mbox{,\,\, then }\,\,Y'\in  \cY_-.
$$
A {\it section} of the action is a semi-section such that $\cY_0=\emptyset$ and $\cY_-,\cY_+\neq \emptyset$.
\end{definition}

The following statement is a reformulation of \cite[Theorem~2.1]{BBS1}:

\begin{proposition}\label{prop:semisection}
Let $X$ be a normal projective variety admitting a $\C^*$-action as above. Given a semi-section $\cY=\cY_-\sqcup \cY_0\sqcup\cY_+$ of the action, and denoting by $U$ the open set $X\setminus \big(\bigcup_{Y\in \cY_+}X^-(Y)\cup \bigcup_{Y\in \cY_-}X^+(Y)\big)$, there exists a semi-geometric quotient  of $U$ by the induced action of $\C^*$. Furthermore, if the semi-section is a section, then the quotient of $U$ by $\C^*$ is geometric.
\end{proposition}

In this paper we will consider only a certain type of sections and semi-sections, whose quotients will not only be proper, but projective, since they will be standard GIT quotients of $X$. The construction is the following. 

\begin{construction}\label{cons:sections} Let $(X,L)$ be a $\Q$-polarized pair with a nontrivial $\C^*$-action, and denote by $0=a_0<\dots<a_r=\delta$  
the critical values of the action. We obtain a semi-section (respectively a section) of the action choosing an index $i\in\{0,\dots ,r\}$ (resp. $i\in\{0,\dots ,r-1\}$), and setting 

$$\begin{array}{l}\cY_-:=\{Y\in\cY|\,\, \mu_L(Y)<a_i\},\quad\cY_0:=\{Y\in\cY|\,\, \mu_L(Y)=a_i\},\\[2pt]\cY_+:=\{Y\in\cY|\,\, \mu_L(Y)>a_i\},\end{array}$$ 

\noindent (resp. $\cY_-:=\{Y\in\cY|\,\, \mu_L(Y)\leq a_i\}$, $\cY_+:=\{Y\in\cY|\,\, \mu_L(Y)\geq a_{i+1}\}$). Let us denote by $X^{\ss}(i,i)$ (resp. $X^{\ss}(i,i+1)$) the open set $X\setminus \big(\bigcup_{Y\in \cY_-}X^+(Y)\cup \bigcup_{Y\in \cY_+}X^-(Y)\big)$, and by $\GX(i,i)$ (resp. $\GX(i,i+1)$) the corresponding proper semi-geometric (resp. geometric) quotients. See Figure \ref{fig:sinksource} below.
\begin{figure}[h!]
\includegraphics[width=12cm]{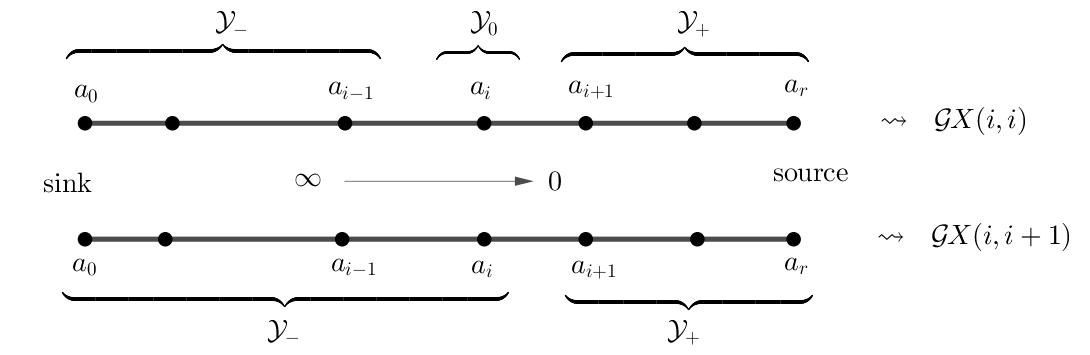}
\caption{Weight representation of the semi-geometric and geometric GIT quotients of $X$.}
\label{fig:sinksource}
\end{figure}
\end{construction}

\begin{definition}\label{def:Atau}
For every rational number $\tau\in\QQ\cap[0,\delta]$, we define a graded algebra and a homogeneous ideal 
$$I_\tau:=\bigoplus_{\substack{m>0\\m\tau\in\Z}}\HH^0(X,mL)_{m\tau}\subset A_\tau:=\bigoplus_{\substack{m\geq 0\\m\tau\in\Z}}\HH^0(X,mL)_{m\tau}\subset A:=\bigoplus_{\substack{m\geq 0}}\HH^0(X,mL),$$ 
where the subindex $m\tau$ denotes the direct summand of $\HH^0(X,mL)$ of $\C^*$-weight equal to $m\tau$, and $\HH^0(X,mL)$ is set to be zero if $mL\notin\Pic(X)$. 
\end{definition}

The construction of the varieties $\GX(i,i)$, $\GX(i,i+1)$ can be then described in terms of the original linearization of $L$ by means of Mumford's Geometric Invariant Theory (cf. \cite[Amplification~1.11, p.40]{MFK}), as follows. 

\begin{proposition}\label{prop:semisection2}
Let $(X,L)$ be a $\Q$-polarized pair with a nontrivial $\C^*$-action with  critical values $0=a_0<\dots<a_r=\delta$. Then:
\begin{itemize}
\item for every $i=0,\dots,r$, we have $$X^{\ss}(i,i)=X\setminus V(I_{a_i}\otimes_{A_{a_i}} A),\quad\GX(i,i)=\Proj(A_{a_i});$$
\item for every $i=0,\dots,r-1$, and every $\tau\in(a_i,a_{i+1})\cap \Q$, we have $$X^{\ss}(i,i+1)=X\setminus V(I_{\tau}\otimes_{A_\tau} A),\quad\GX(i,i+1)=\Proj(A_\tau).$$
\end{itemize}
\end{proposition}

\begin{definition}\label{def:quotients}
In the above situation, we will call the varieties $\GX(i,i)$, $\GX(i,i+1)$ the {\em GIT-quotients of the pair $(X,L)$ by the action of $\C^*$}.
\end{definition}

\begin{remark}\label{rem:extremal-semigeometric}
Note that, since the intersection of all the open sets $X^{\ss}(i,i+1)$ is nonempty, the varieties $\GX(i,i+1)$ are birationally equivalent. Moreover, the natural birational maps among them fit in the following commutative diagram, whose diagonal arrows are contractions:

$$
\begin{tikzcd}[
  column sep={3.9em,between origins},
  row sep={3.5em,between origins},
]
&\GX(0,1) \arrow[rr,dashed] \arrow[rd] \arrow[dl]&&\GX(1,2) \arrow[dl]&\dots\ \ &\GX(r\!-\!2,r\!-\!1) \arrow[rr,dashed] \arrow[rd] &&\GX(r\!-\!1,r)\arrow[rd] \arrow[dl]&\\
\GX(0,0)&&\GX(1,1)&\dots&&\dots\ \ &\GX(r\!-\!1,r\!-\!1)&&\GX(r,r)
\end{tikzcd}
$$ 
Note also that, by construction, $\GX(0,0)=Y_0$, $\GX(r,r)=Y_r$, and, since $X^{\pm}(Y)\simeq \cN^{\pm} (Y)$ for every $Y\in \cY$, the fibers of the diagonal morphisms are weighted projective spaces (standard projective spaces if the action is equalized). 
\qed
\end{remark}

We observe also that the equalization hypothesis implies that the geometric quotients of $X$ are $\C^*$-principal bundles.

\begin{lemma}\label{lem:principal}
Let $(X,L)$ be a $\Q$-polarized pair with a nontrivial equalized $\C^*$-action with  critical values $0=a_0<\dots<a_r=\delta$. The geometric quotients $\pi_{i,i+1}:X^{\ss}(i,i+1)\to\GX(i,i+1)$, $i=0,\dots,r-1$, are $\C^*$-principal bundles. In particular, if $X$ is smooth, then its geometric quotients by $\C^*$ are smooth.
\end{lemma}

\begin{proof}
We take an affine open covering $\{U_i\}$ of $\GX(i,i+1)$ such that the inverse image of every $U_i$ is an affine $\C^*$-scheme. It is enough to show that $\pi_i^{-1}(U_i)\to U_i$ is a $\C^*$-principal bundle for every $i$, and since by Lemma \ref{lem:equalized} we know that the action of $\C^*$ has trivial stabilizers, this is a Corollary of Luna Slice Theorem (see \cite[Corollary~in~p.199]{MFK}).
\end{proof}

In particular, one may construct a different compactification of $X^{\ss}(i,i+1)$ by considering the natural action of $\C^*$ on $\P^1$ and considering the variety:
$$
\widehat{X}(i,i+1):=X^{\ss}(i,i+1)\times^{\C^*}\P^1=X^{\ss}(i,i+1)\times\P^1/\sim,
$$
where $(x,\lambda)\sim (x',\lambda')$ if and only if $x'=tx$, $\lambda'=t\lambda$ for some $t\in\C^*$. 
Intuitively $\widehat{X}(i,i+1)$ is constructed by adding to $X^{\ss}(i,i+1)$ two sections corresponding to the limit points of the action when $t$ goes to $0$ and infinity. The variety $\widehat{X}(i,i+1)$ is  a $\P^1$-bundle, projectivization of a decomposable rank two vector bundle on $\GX(i,i+1)$, and it is birationally equivalent to $X$ by construction. The following statement describes when the natural map $X\dashrightarrow \widehat{X}(i,i+1)$ is a small modification:

\begin{proposition}\label{prop:P1model} 
Let $(X,L)$ be a $\Q$-polarized pair with a nontrivial equalized B-type $\C^*$-action with  critical values $0=a_0<\dots<a_r=\delta$, and let $i\in\{0,\dots,r-1\}$ satisfy:
$$\begin{array}{l}
\nu^-(Y)>1 \mbox{ for every inner fixed component } Y\mbox{ s.t }\,\mu_L(Y)\leq a_i,\\[2pt]
\nu^+(Y)>1 \mbox{ for every inner fixed component } Y\mbox{ s.t }\,\mu_L(Y)\geq a_{i+1}.
\end{array}\eqno{(\dagger)}
$$
Then $X$ is $\C^*$-equivariantly isomorphic in codimension one to 
$\widehat{X}(i,i+1)$. 
\end{proposition}

\begin{proof}
Note first that the rational map $X\dashrightarrow \widehat{X}(i,i+1)$ is defined on $X^{\ss}(i,i+1)$, and, by the B-type hypothesis, it is also defined on every point of $Y_0\cup Y_r$ that lies in the boundary of an orbit contained in $X^{\ss}(i,i+1)$. In other words, the map is defined in 
$$
X\setminus \left(\bigcup_{\substack{\mu_L(Y)\leq a_i\\Y\text{ inner}}}{X^{+}(Y)}\cup\bigcup_{\substack{\mu_L(Y)\geq a_{i+1}\\Y\text{ inner}}}{X^{-}(Y)} \right).
$$
Since the hypothesis $(\dagger)$ implies that each of the varieties $\overline{X^{\pm}(Y)}$ has codimension at least two (see \cite[Theorem~4.2]{CARRELL}), it follows that $X\dashrightarrow \widehat{X}(i,i+1)$ is defined in codimension one.
\end{proof}

%!TEX root = WORS3.tex

\section{Small modifications of B-type actions and bordisms}\label{sec:bord}

In this section we consider the case of a $\Q$-polarized pair $(X,D)$, with $X$ smooth, admitting a B-type equalized action, and we will show how to construct, under some assumptions, two $\C^*$-equivariant small $\Q$-factorial modifications of $X$ which are smooth and of smaller criticality, $\SX^-$, $\SX^+$, having indeterminacy locus $B^+_1$, $B^-_{r-1}$, respectively. 
The construction is an extension of the one described in \cite[Section~8]{WORS1} in the particular case of bandwidth three actions. The main result we will show is the following:

\begin{theorem}\label{thm:flip} Let $X$ be a smooth projective variety and $D$ be an ample $\Q$-divisor on $X$, such that $(X,D)$ admits a B-type equalized action of $\C^*$, of criticality $r\geq 2$ and critical values $0,a_1,\dots,a_r$. Assume moreover that $\nu^-(Y)>1$ for every fixed point component $Y$ of weight $a_1$. Then there exists a smooth variety $\SX^-$ together with a B-type equalized $\C^*$-action, and a $\C^*$-equivariant small modification $\psi^-:X\dashrightarrow\SX^-$, such that:
\begin{itemize}
\item the indeterminacy locus of $\psi^-$ is $B^+_1$;
\item  the $\Q$-divisor $D^-_\tau:=\psi^-_*(D-\tau Y_0)$ is ample for every $\tau\in (a_1,a_2)\cap \Q$;
\item $(\SX^-,D^-_\tau)$ has critical values $0,a_2-\tau,\dots,a_r-\tau$, and criticality $r-1$.
\end{itemize}
Moreover, the map $\psi^-$ sends isomorphically the fixed point components in $X$ of weight $a_i$, $i\geq 2$, to the fixed point components of weight $a_i-\tau$ in $\SX^-$, and preserves $\nu^\pm$ on these components.
\end{theorem}

\begin{remark}\label{rem:flip}
Note that, by composing the action of $\C^*$ with the inversion map $t\mapsto t^{-1}$ (an operation that exchanges the sink and the source of the action), the above statement implies that if $\nu^+(Y)>1$ for every fixed point component of weight $a_{r-1}$, there exists a small  $\C^*$-equivariant modification $\psi^+:X\dashrightarrow\SX^+$ onto a smooth variety admitting a B-type equalized action, whose indeterminacy locus is $B^-_{r-1}$, such that the $\Q$-divisor $D^+_\tau:=\psi^+_*(D-(\delta-\tau) Y_r)$ is ample for every $\tau\in (a_{r-2},a_{r-1})\cap \Q$, and such that the critical values of $(\SX^+,D^+_\tau)$ are $0,a_1-\tau,\dots,a_{r-2}-\tau$.
\end{remark}

The situation of Theorem \ref{thm:flip} may be represented as follows. We consider in $\NU(X)$ the affine plane $\cS_D$ of classes of the form $D-\tau_-Y_0-(\delta-\tau_+)Y_r$, $\tau_-,\tau_+\in\R$. The region $\Nef(X)\cap \cS_D$ meets $\Nef(\SX^-)\cap\cS_D$ on an edge passing by $D^-_{a_1}=D-a_1Y_0$ and $\Nef(\SX^+)\cap\cS_D$ on an edge passing by $D^+_{a_{r-1}}=D-(\delta-a_{r-1})Y_r$. Each wall-crossing corresponds to a flip. We have represented the intersections of the cones with $\cS_D$ in Figure \ref{fig:flip}; a more complete description will be provided in Section \ref{sec:movbordism}.

\begin{center}
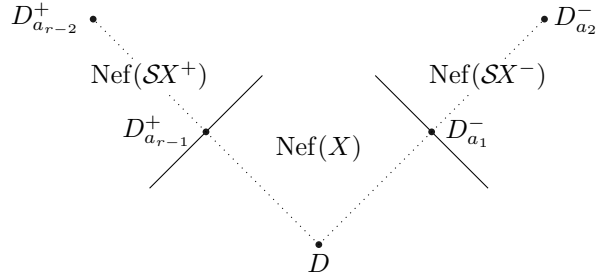
\begin{figure}[h!]
\begin{tikzpicture}[scale=0.75]
\draw (-3,-3) -- (-1,-1);
\draw (1,-1) -- (3,-3);
\fill[black!90!white] (0,-4) circle (0.6mm);
\fill[black!90!white] (2,-2) circle (0.6mm);
\fill[black!90!white] (4,0) circle (0.6mm);
\fill[black!90!white] (-2,-2) circle (0.6mm);
\fill[black!90!white] (-4,0) circle (0.6mm);
\draw[dotted] (0,-4) -- (2,-2);
\draw[dotted] (2,-2) -- (2.7,-1.3);
\draw[dotted] (3.3,-0.7) -- (4,0);
\draw[dotted] (0,-4) -- (-2,-2);
\draw[dotted] (-2,-2) -- (-2.7,-1.3);
\draw[dotted] (-3.3,-0.7) -- (-4,0);
\node [below] at (0,-4) {$D$};
\node [left] at (-2.1,-2) {$D^+_{a_{r-1}}$};
\node [left] at (-4,0) {$D^+_{a_{r-2}}\,$};
\node [right] at (2,-2) {$\,D^-_{a_1}$};
\node [right] at (4,0) {$D^-_{a_2}\,$};
\node  at (0,-2.3) {$\Nef(X)$};
\node  at (-3,-1) {$\Nef(\SX^+)$};
\node  at (3,-1) {$\Nef(\SX^-)$};
\end{tikzpicture}
\caption{Two flips of $X$.\label{fig:flip}} 
\end{figure}
\end{center}

The proof will be done in three steps. We will start by showing in Section \ref{ssec:existence}, by means of Nakano Contractibility Theorem, that $\SX^-$ exists as a compact complex manifold (Corollary \ref{cor:SQM1}). Then we will proof  in Section \ref{ssec:projectivity} that $\SX^-$ is projective, by analyzing the restriction of  $\psi^-:X\dashrightarrow \SX^-$ to the extremal fixed point components of $X$ and relating them to the geometric quotients of $X$. Finally, in Section \ref{ssec:properties}, we will check that $\SX^-$ satisfies all the properties required in the statement.

\subsection{Existence of $\SX^-$ as a complex manifold}\label{ssec:existence}

Let us start by studying the variety $B^+_1$, which will be exceptional locus of  $\psi^-:X\dashrightarrow \SX^-$, and proving that it is a projective bundle over $Y_1$, in the sense that the irreducible components of $B^+_1$ are projective bundles over the irreducible components of $Y_1$. Note that $Y_1$ is not assumed to be irreducible, and in fact it may have irreducible component of different dimensions (see Remark \ref{rem:B1irred} below). We will also identify the normal bundle to $B^+_1$ in $X$.  

\begin{proposition}\label{prop:SQM1}
Let $X$ be a smooth projective variety, and $D$ be an ample $\Q$-divisor on $X$, such that $(X,D)$ admits an equalized B-type $\C^*$-action, of criticality $r\geq 2$. Then $B^+_1\subset X$ is the disjoint union of the varieties $B^+(Y)$, $Y\in\cY$, $\mu_D(Y)=a_1$, and each $B^+(Y)$ is a projective bundle over $Y$, of relative dimension $\nu^+(Y)$. 
Moreover, denoting by $p:B^+(Y)\to Y$ the natural projection, and by $F$ a fiber of $p$ we have:
$$\cN_{B^+(Y)|X}|_F \simeq \cO_{F}(-1)^{\oplus \nu^-(Y)}.$$
\end{proposition}

\begin{remark}\label{rem:B1irred}
One may  construct examples in which $Y_1$ has components of different dimensions. For instance, the variety $X=\mathbb{P}^1\times Q^{n-1}$ (where $Q^{n-1}$ denotes a smooth $(n-1)$-dimensional quadric) admits an equalized $\C^*$-action of bandwidth $3$ such that $Y_0$ and $Y_3$ are two isolated points, and both  $Y_1$ and $Y_2$ are  the union of a point and $Q^{n-3}$ (see \cite[Theorem 3.6 (2)]{RW}). If we consider the blowup of $X$ along $Y_0$ and $Y_3$ we obtain an equalized B-type $\C^*$-action  with $B^{+}_1$, $B^{+}_2$ still reducible with two irreducible components. 
\end{remark}

\begin{proof}[Proof of Proposition \ref{prop:SQM1}]
Note that, since $X$ is smooth and the action is of B-type, then, by Bia{\l}ynicki-Birula Theorem (see for instance \cite[Theorem~2.8]{WORS1}) $X^-(Y_0)$ is isomorphic to a line bundle over $Y_0$; in particular, two components of the form $B^+(Y)\subset B^+_1$ do not intersect and we may write 
$$B^+_1=\bigsqcup_{\substack{Y\in \cY\\\mu_D(Y)=a_1}}B^+(Y).$$

Let $Y\subset Y_1$ be an irreducible component.  
Consider  the set $X^+(Y)\setminus Y \subset X$
which is isomorphic, by the Bia{\l}ynicki-Birula decomposition (see Section \ref{ssec:actions}), to the open set $$\cN^+(Y)\setminus s_0(Y)\subset \cN^+(Y),$$ where $s_0$ denotes the zero section of $\cN^+(Y)$. On the other hand, a similar argument shows that $X^+(Y)\setminus Y$ is isomorphic to the open set 
$$\cN^-(Y_0)_{|B^+(Y)\cap Y_0}\setminus s'_0(B^+(Y)\cap Y_0)\subset \cN^-(Y_0)_{|B^+(Y)\cap Y_0},$$ 
where $s'_0$ denotes the zero section of $\cN^-(Y_0)$. This is a $\C^*$-principal bundle over $B^+(Y)\cap Y_0$, so  its projection to $B^+(Y)\cap Y_0$ is a geometric quotient, whose fibers are the orbits of the action of $\C^*$.  Since the action is  equalized, $\C^*$ acts homothetically on $\cN^+(Y)$, therefore we get 
\begin{equation}\label{eq:excepY0}
B^+(Y)\cap Y_0\simeq\P(\cN^+(Y)^\vee).
\end{equation}

Note that $B^+(Y)$ can be written as the union of two open sets, $X^+(Y)$, and $\cN^-(Y_0)_{|B^+(Y)\cap Y_0}$, and the smoothness of $B^+(Y)$ follows.

 We consider now the blowup $B^{+}(Y)^{\flat}$ of $B^+(Y)$ along $Y$; it is a $\P^1$-bundle over $\P(\cN^+(Y)^\vee)$ with two sections: one is the exceptional divisor of the blowup, and the other is the strict transform of $B^+(Y)\cap Y_0$. By computing the normal bundle of the first section, it follows that:
$$
B^{+}(Y)^{\flat}\simeq\P\left(\cO_{\P(\cN^+(Y)^\vee)}(1)\oplus\cO_{\P(\cN^+(Y)^\vee)}\right).
$$
From this it easily follows that 
\begin{equation}\label{eq:proof}
 B^+(Y)\simeq\P\left(\cN^+(Y)^\vee\oplus\cO_{Y}\right).
\end{equation}

Let $F\simeq \P^{\nu^+(Y)}$ be a fiber of $p$, and let $x_1$ be the intersection point $F\cap Y$. 
We will show that $(\cN_{B^+(Y)|X})_{|\P^1}\simeq\cO_{\P^1}(-1)^{\oplus \nu^-(Y)}$ for every line $\P^1\subset F$ passing by $x_1$; the results will then follow from \cite[Theorem~3.2.1]{OSS}.

Given such a line $\P^1$, we set $x_0:=\P^1\cap Y_0$, and compute the splitting type of the bundle $(\cN_{B^+(Y)|X})_{|\P^1}$ by using the vector bundle version of the AM vs FM formula (\cite[Lemma~2.17]{WORS1}). Since the action is equalized, the splitting type can be obtained as the differences of the weights of the action on the vector spaces $\cN_{B^+(Y)|X,x_1}$, $\cN_{B^+(Y)|X,x_0}$, and it is a straightforward computation that these weights are, respectively $\big((-1)^{\nu^-(Y)}\big)$, $\big(0^{\nu^-(Y)}\big)$,   with the exponent denoting the occurrence of $-1$ and $0$, respectively. This completes the proof.
\end{proof}

Let us now denote by $b:X'\to X$ the blowup of $X$ along $B^+_1$ and by $E$ its exceptional divisor, which is the disjoint union of the divisors $E_Y=\P(\cN_{B^+(Y),X}^\vee)$, $Y\in \cY$, $\mu_D(Y)=a_1$; recall also the varieties $V^\pm(Y)$, for every $Y\in \cY$, defined in (\ref{eq:Vpm}). The following statement describes $E$ and its irreducible components.

\begin{proposition}\label{prop:fornakano} 
Let $(X,D)$ be as in Proposition \ref{prop:SQM1}. Then: 
$$E\simeq B^+_1\times_{Y_1}V^-_1.$$
\end{proposition}

\begin{proof}
We will show that for every irreducible component $Y\subset Y_1$ we have:
$$E_Y \simeq B^+(Y)\times_Y V^-(Y).$$

Using Proposition \ref{prop:SQM1}, given a fiber $F$ of $p:B^+(Y)\to Y$, we have that $\cN_{B^+(Y)|X}|_F \simeq \cO_{F}(-1)^{\oplus \nu^-(Y)}$, hence $E_Y$ is a $\P^{\nu^+(Y)}$-bundle over a variety $Z$, and $Z$ is necessarily a $\P^{\nu^-(Y)-1}$-bundle over $Y$, i.e. we have a Cartesian square:
$$
\xymatrix@C=15mm@R=12mm{E_Y\ar[r]^(.60){\P^{\nu^+(Y)}}\ar[d]_{\P^{\nu^-(Y)-1}}&Z\ar[d]^{\P^{\nu^-(Y)-1}}\\B^+(Y)\ar[r]^(.60){\P^{\nu^+(Y)}}&Y}
$$
whose maps are projective bundles. Now we note that, since $B^+(Y)$ is isomorphic to $X^+(Y)$ in a neighborhood of $Y$, then $(\cN_{B^+(Y)|X})_{|Y}\simeq\cN^-(Y)$, and we have a diagram composed of two Cartesian squares:
$$
\xymatrix@C=10mm@R=10mm{V^-(Y)\ar[r]\ar[d]&E_Y\ar[r]^{\bar p}\ar[d]&Z\ar[d]\\Y\ar@{^(->}+<2.5ex,0ex>;[r]&B^+(Y)\ar[r]^p&Y}
$$
In particular, the composition of the two upper horizontal arrows is an isomorphism, that is 
$Z\simeq V^-(Y)$, and the statement follows.
\end{proof}

Since $B^+_1$ is $\C^*$-invariant, then the action of $\C^*$ on $X$ extends to $X'$. This action is of B-type by construction, and its sink which we denote by $Y_0'$ is the strict transform of $Y_0$ via $b$. Since $B^+_1$ intersects $Y_0$ transversally, it follows that $Y_0'$ is the blowup of $Y_0$ along $B^+_1\cap Y_0$. We have shown in the proof of Proposition \ref{prop:SQM1} (see Equation (\ref{eq:excepY0})) that for each component $Y\subset Y_1$ we have that $B^+(Y)\cap Y_0\simeq \P\left(\cN^+(Y)^\vee\right)=V^+(Y)$. We will describe now the exceptional locus $E_Y\cap Y_0'$ of $b_{|Y_0'}$.

\begin{corollary}\label{cor:fornakano} 
Let $(X,D)$ be as in Proposition \ref{prop:SQM1}, let $Y\subset Y_1$ be an irreducible component of $Y_1$, and $E_Y$ be the corresponding irreducible component of the exceptional divisor $E$ of the blowup $b:X'\to X$. Moreover, let $\Gamma\simeq\P^{\nu^+(Y)-1}$ be any fiber of $V^+(Y)\to Y$. Then 
$$
\cN_{V^+(Y)|Y_0}|_{\Gamma}\simeq\cO_{\Gamma}(-1)^{\oplus \nu^-(Y)},\quad
E_Y\cap Y_0'\simeq\P(\cN^\vee_{V^+(Y)|Y_0})\simeq V^+(Y)\times_Y V^-(Y).
$$
\end{corollary}

\begin{proof}  
Under the identification $B^+(Y)\cap Y_0=V^+(Y)$  we may write:
$$(\cN_{B^+(Y)|X})_{|V^+(Y)}\simeq \cN_{V^+(Y)|Y_0}.$$ Since ${\Gamma}$ is a projective subspace of a fiber of $B^+(Y)\to Y$, the first equality holds. 
Projectivizing the above isomorphism, the inclusion $\P(\cN^\vee_{V^+(Y)|Y_0})\subset\P(\cN_{B^+(Y)|X}^\vee)$ fits into a commutative diagram consisting of two Cartesian squares:
$$
\xymatrix@C=15mm{\P(\cN^\vee_{V^+(Y)|Y_0})\ar@{^(->}+<8ex,0ex>;[r]\ar[d]&\P(\cN_{B^+(Y)|X}^\vee)\ar[r]\ar[d]&V^-(Y)\ar[d]\\V^+(Y)\ar@{^(->}+<5ex,0ex>;[r]&B^+(Y)\ar[r]^p&Y}
$$
This implies that:
\[
\P(\cN^\vee_{V^+(Y)|Y_0})\simeq V^+(Y)\times_Y V^-(Y).\qedhere
\]
\end{proof}

\begin{corollary}\label{cor:SQM1}
Let $X$ be a smooth projective variety, and $D$ be an ample $\Q$-divisor on $X$, such that $(X,D)$ admits an equalized B-type $\C^*$-action, of criticality $r\geq 2$. Let $X'$ be the blowup of $X$ along $B^+_1$, with exceptional divisor $E$, and let $Y_0'\subset X'$ be the strict transform of $Y_0$ into $X'$. Then there exists a proper map $$b':X'\to \SX^-$$ which is the blowup of a smooth proper variety $\SX^-$ along a subvariety isomorphic to $V^-_1$, whose exceptional divisor is $E$. Moreover:
\begin{itemize}
\item $b'_{|_E}$ is the natural projection $E\simeq B^+_1\times_{Y_1}V^-_1\to V^-_1$.
\item $b'_{|Y'_0}$ is the smooth blowup of $b'(Y_0')$ along $V^-_1$, with exceptional divisor $E\cap Y_0'\simeq V^+_1\times_{Y_1}V^-_1$.
\end{itemize}
\end{corollary}

\begin{proof}
The divisor $E$ is not necessarily connected, but we are not claiming yet that $\SX^-$ is projective, so the statement is local: we will show that there exists a holomorphic contraction $b':X' \to \SX^-$ that is a smooth blowup on an analytic neighborhood of every component $E_Y$ of $E$, mapping $E_Y=B^+(Y)\times_Y V^-(Y)$ onto a variety isomorphic to $V^-(Y)$. Let us denote by $F'\simeq \P^{\nu^+(Y)}$ a fiber of the projective bundle $\bar p:E_Y \to V^-(Y)$. By Nakano Contractibility Criterion (\cite[Theorem~3.2.8]{BS}) for the first part of the statement it is enough to show that 
$$
E_{Y|F'}\simeq \cO_{F'}(-1).$$
This follows by noting that, by Proposition \ref{prop:fornakano}, $F'$ is a section of $E_Y \to B^+(Y)$ over its image, that we denote by $F$, and $F$ is a fiber of  $p:B^+(Y)\to Y$ (see Proposition \ref{prop:fornakano}). By Proposition \ref{prop:SQM1}, $(\cN_{B^+(Y)|X})_{|F}\simeq \cO_F(-1)^{\oplus \nu^-(Y)}$, and so the above equality follows by a straightforward computation. 

For the second part of the statement we note that, by Corollary \ref{cor:fornakano}, $b'$ maps $E_Y\cap Y_0'$ onto a subvariety  of $b'(Y_0')$ isomorphic to $V^-(Y)$, for every component $Y\subset Y_1$. The fact that $b'_{|Y_0'}$ is again a smooth blowup follows again by Nakano's Criterion and by the first part of Corollary \ref{cor:fornakano}.
\end{proof}

Let us define the rational map $\psi^-: X \dashrightarrow \SX^-$ as follows:
$$
\psi^-:=b'\circ b^{-1}.
$$
Its indeterminacy locus is $B^+_1$ by construction. We will now prove that the $\C^*$-action on $X'$ descends to $\SX^-$, and identify the corresponding sink. 

\begin{corollary}\label{cor:SX-}
The map $b':X'\to\SX^-$ induces a B-type equalized $\C^*$-action on $\SX^-$, whose sink is isomorphic to $\GX(1,2)$.
\end{corollary}

\begin{proof}
The restriction of the $\C^*$-action on $X'$ to every component $E_Y=B^+(Y)\times_Y V^-(Y)$ is the fiber product of the natural action of $\C^*$ on $B^+(Y)$ and $V^-(Y)$. In particular, the fibers of $b'$ are $\C^*$-invariant, and the action descends to $\SX^-$. 

Note that a priori we do not know if $\SX^-$ is projective; by saying that the action of $\C^*$ on it is of B-type and equalized we mean that the sink and the source of the action are divisors, and that the isotropy group of every point in a $1$-dimensional orbit is trivial. 

The B-type property is clear since the extremal fixed point components of $\SX^-$ are $\psi^-(Y_0)$ and $\psi^-(Y_r)$. In order to see that the action is equalized we observe that the image of $E$ into $\SX^-$ is contained in the image of $Y_0'$ (see Corollary \ref{cor:SQM1}); this implies that the $1$-dimensional orbits of $\SX^-$ do not meet $b'(E)$, and so they are mapped isomorphically into $X$. In particular, the isotropy groups of the points in these orbits are trivial. 

In order to study the sink $\psi^-(Y_0)$ of $\SX^-$, we note first that the open subset $X^{\ss}(1,2)\subset X$ maps isomorphically into $X'$ and $\SX^-$. By construction, the image $\psi^-(X^{\ss}(1,2))$ contains all the $1$-dimensional orbits in $\SX^-$ whose limit at $\infty$ belongs to $\psi^-(Y_0)$; more precisely, we claim that for every point $P\in \psi^-(Y_0)$ there exists a unique orbit in $\psi^-(X^{\ss}(1,2))$ whose limit at $\infty$ is $P$. This is obvious if $P\in \psi^-(X\setminus B^+_1)$, so we are left to study the case in which $P$ belongs to the complementary of this set in $\psi^-(Y_0)$, which is isomorphic to $V^-_1=\P(\cN^-(Y_1)^\vee)$. Denoting by $p:V^-_1\to Y_1$ the natural projection, given a point $P\in V^-_1=\P(\cN^-(Y_1)^\vee)$, by the Bia{\l}ynicki-Birula decomposition, there exists a unique orbit in $X^{\ss}(1,2)$ whose limit at $\infty$ is the point $p(P)$ with tangent direction $P$. This completes the proof of the claim.  

In particular, we have a bijective morphism from the quotient $\GX(1,2)$ to $\psi^-(Y_0)$, sending every orbit in $X^{\ss}(1,2)$ to the limit at $\infty$ of its image in $\SX^-$; this map is bijective by the above claim. Since $\psi^-(Y_0)$ is smooth (by Corollary \ref{cor:SQM1}) and $\GX(1,2)$ is normal, the map is an isomorphism.  
\end{proof}

\subsection{Projectivity of $\SX^-$}\label{ssec:projectivity}

Let $b:X'\to X$ be the blowup of $X$ along $B^+_1$,  $E$ the exceptional divisor and $Y_0'$ the strict transform of $Y_0$.
In order to prove that the holomorphic map $b':X'\to \SX^-$ provided by Corollary \ref{cor:SQM1} is  projective, we will show that it admits a supporting $\Q$-divisor, i.e., there exists a nef $\Q$-divisor  on $X'$ having intersection number zero exactly on the curves contracted by $b'$:

\begin{proposition}\label{prop:suppdiv}
Under the assumptions of Theorem \ref{thm:flip}, the $\Q$-divisor
$$
D':=b^*(D-\tau Y_0)-(\tau-a_1)E,\quad \tau\in(a_1,a_2)\cap \Q,
$$ is a supporting divisor of $b'$.
\end{proposition}

The proof of this fact is done in two steps. First we show (Lemma \ref{lem:suppdiv}) that the statement can be reduced to analyzing the behavior of this divisor in $Y_0'$; then the proof is concluded by interpreting $b'_{|Y_0'}$ as a resolution of the natural map among two geometric quotients of $X$.

\begin{lemma}\label{lem:suppdiv} If $D'_{|Y_0'}$ is a supporting $\QQ$-divisor for $b'_{|Y_0'}$, then $D'$ is a supporting $\QQ$-divisor for $b'$.
\end{lemma}

\begin{proof}
The Mori cone of $X'$ is generated by the numerical classes of $\C^*$-invariant curves (cf. \cite[Lemma~1.5]{RW}), that is by the classes of closures of $1$-dimensional orbits and by the classes of curves contained in the fixed point components $Y_i'$. Let us note that $Y_i'=Y_i$ for $i \ge 2$ and $Y_1'=V_1^-$. 
It is easy to show, by using Lemma \ref{lem:AMFM}, that $D' \cdot C >0$ if $C$ is:
\begin{itemize}
\item the closure of an orbit having sink on $Y'_i$, $i\geq 1$; in this case $D'\cdot C=b^*D\cdot C>0$ if $i \ge 2$ and $D'\cdot C\geq(a_2-a_1)+(a_1-\tau)>0$ if $i=1$;
\item the closure of an orbit  having  sink on $Y'_0$, and source on $Y'_i$, $i\geq 2$, for which $D'\cdot C=a_i-\tau>0$;
\item a curve contained in $Y'_i$, $i\geq 2$, for which we have $D'\cdot C=b^*D\cdot C>0$.  
\end{itemize}
Moreover $D' \cdot C=0$ if $C$ is  is the closure of an orbit joining $Y'_0$ with $Y'_1$, since, by
Corollary \ref{cor:SQM1}, $E\cdot C=-1$, hence $D'\cdot C=a_1-\tau-(\tau-a_1)(-1)=0$.

We are left to study the classes of curves contained in $Y_1'=V^-_1$. 
We have shown (Corollary \ref{prop:fornakano}) that the exceptional divisor $E\subset X'$ is a projective bundle over $V^-_1$, containing $V^-_1$ as a section. If $C\subset V^-_1$ were an irreducible curve such that $D'\cdot C\leq 0$, we could find an irreducible curve $C'$ in $Y_0' \cap E$ such that, for some positive integer $m$ the cycle $C'-mC$ would be numerically proportional to the class $[\ell^-]$ of a line in a fiber of $E\to V^-_1$.
 Since $D'\cdot \ell^-=0$,  we would have $D'\cdot C'\leq 0$, hence $D' \cdot C'=0$ and it would follow that $C'$ is numerically proportional to $\ell^-$. Hence also $C$ would be numerically proportional to $\ell^-$, and therefore contracted by the projection $E\to V^-_1$, a contradiction.
\end{proof}

\begin{proof}[Proof of Proposition \ref{prop:suppdiv}]
By means of Corollary \ref{cor:SX-}, let us identify the sink of $\SX^-$ with $\GX(1,2)$. Consider the rational map $\psi^-_{|Y_0}:Y_0\dashrightarrow \GX(1,2)\subset \SX^-$. As shown in the proof of Corollary \ref{cor:SX-}, this is the natural map among the geometric quotients $Y_0\simeq\GX(0,1)$, $\GX(1,2)$. In particular, it coincides with the restriction of the quotient map $X\dashrightarrow \GX(1,2)$. 

Since the action of $X$ is of B-type (so its blowup along its sink and source is an isomorphism), by applying Proposition \ref{prop:semisection2} and Lemma \ref{lem:projections2}, the quotient map $X\dashrightarrow \GX(1,2)$ is given by a linear system of the form $|m(D-\tau Y_0-(\delta-\tau)Y_r)|$, for $\tau\in \Q\cap(a_1,a_2)$, and $m$ large enough and divisible. Therefore $\psi^-_{|Y_0}$  is given by a linear subsystem of $|m(D-\tau Y_0)_{|Y_0}|$.

Since this rational map is resolved via the blowup $b$, a supporting divisor of $b':Y_0'\to \GX(1,2)$ will be of the form 
$$
b^*((D-\tau Y_0)_{|Y_0})-r E_0,
$$
for some $r\in \Q$. Imposing that this divisor has degree zero on a curve contracted by $b'$, we easily get $r=\tau-a_1$. Then the statement follows by Lemma \ref{lem:suppdiv}.
\end{proof}

\subsection{The fixed point components of $\SX^-$}\label{ssec:properties} 

First of all we note that, from the results in Sections \ref{ssec:existence}, \ref{ssec:projectivity}, we already know that the rational map $\psi^-:X\dashrightarrow \SX^-$ has indeterminacy locus $B^+_1$, and that it is $\C^*$-equivariant with respect to an action on $\SX^-$ that is of B-type and equalized. From Proposition \ref{prop:suppdiv} and applying $b'_*$, we obtain that $D^-_\tau:=\psi^-_*(D-\tau Y_0)$ is ample on $\SX^-$, for every $\tau\in \Q\cap (a_1,a_2)$. Then we are left with studying the fixed point components of $\SX^-$, and their weights with respect to $D^-_\tau$.

Note that there exists a linearization of the $\C^*$-action on the line bundle $\cO_X(Y_0)$ satisfying that $\mu_{\cO_X(Y_0)}(Y_0)=0$. Then, by Lemma \ref{lem:AMFM}, we have that $\mu_{\cO_X(Y_0)}(Y)=1$ for any fixed point component $Y\neq Y_0$.

The fixed components in $\SX^-\setminus \psi^-(Y_0)$ correspond isomorphically via $\psi^-$ to the fixed components in $X\setminus B^+_1$, that is to $\{Y\in\cY|\mu_D(Y)>a_1\}$, and the ranks $\nu^\pm$ are obviously preserved by this correspondence. By definition, on each of these components we have 
$$\mu_{D^-_\tau}(\psi^-(Y))=\mu_{D-\tau Y_0}(Y)=\mu_D(Y)-\tau.$$ 
On the other hand the $D^-_\tau$-weight of the action at $\psi^-(Y_0)$ can be computed at its general point, for which 
$$\mu_{D^-_\tau}(\psi^-(Y_0))=\mu_{D-\tau Y_0}(Y_0)=\mu_D(Y_0)=0.$$
This finishes the proof of Theorem \ref{thm:flip}.

%!TEX root = WORS3.tex

\section{Equalized actions and Mori dream spaces}\label{sec:movbordism}

In this section we will show that, under some assumptions, the blowup of a smooth projective variety $X$ supporting an equalized $\C^*$-action is an MDS; we will study its movable cone, as well as its small $\Q$-factorial modifications. 
Although our arguments should work in broader generality, for the sake of clarity we will stick to the case in which the Picard number of $X$ is one. We will not assume that the sink and the source of the action are positive dimensional; in that particular case, our arguments provide a proof for Theorem \ref{thm:main}. From now on we will work under the following assumptions:

\begin{setup}\label{ass:picZ}
$X$ is a smooth projective variety of Picard number one, $L\in\Pic(X)$ is ample, and there exists an equalized $\C^*$-action on $X$ together with a linearization of the action on $L$. We will assume that $X$ is not $\P^n$ endorsed with the faithful $\C^*$-action that fixes a point and a disjoint hyperplane. 
We will denote by $0=a_0<\dots <a_r=\delta$  the weights of the action on the fixed point components, by $\beta:X^\flat\to X$ the blow-up of $X$ along the sink $Y_0$ and the source $Y_r$, and by $Y^\flat_0,Y^\flat_r\subset X^\flat$ the corresponding exceptional divisors. 
Following \cite[Lemma~3.10]{WORS1}, the action of $\C^*$ on $X$ extends equivariantly via $\beta$ to a B-type action on $X^\flat$. 
\end{setup} 

\begin{remark}\label{rem:RC}
We note that a variety $X$ as in Setup \ref{ass:picZ} is rationally connected (see \cite[Remark~2.3]{WORS1}), in particular $\Pic(X)$ is torsion free, and $\HH^1(X,\cO_X)=0$. 
\end{remark}

\begin{remark}\label{rem:extremaldiv}
Saying that a variety $X$ with Picard number one is not $\P^n$ together with the $\C^*$-action with a fixed point and a fixed hyperplane is equivalent to say that neither the sink nor the source of the action are divisors (cf. \cite[Lemma~6.3]{WORS2}). 
Note also that in the excluded case, $X^\flat$ is the blowup of $\P^n$ along a point, which does not have nontrivial small $\Q$-factorial modifications.
\end{remark}

Before presenting our description of the birational geometry of $X^\flat$, we will need some preliminary results on the fixed point components of $X$. Let us start by recalling the following statement (cf. \cite[Lemma~2.6(2), Lemma~3.10]{WORS1}):

\begin{lemma}\label{lem:bordism2}
Let $(X,L)$ be as in Setup \ref{ass:picZ}. Then the $\C^*$-action on $X^\flat$ is a bordism if and only if $\dim(Y_0),\dim(Y_r)>0$. 
\end{lemma}

In the case in which the extended action on $X^\flat$ is not a bordism, the fixed point components $Y$ with $\nu^{\pm}(Y)=1$ are the ``closest'' to the sink and the source:

\begin{lemma}\label{lem:nonbordism}
Let $(X,L)$ be as in Setup \ref{ass:picZ}, and assume that $Y_0$ (resp. $Y_r$) is a point. Then $Y_1$ is irreducible and $\nu^+(Y_1)=1$ (resp. $Y_{r-1}$ is irreducible and $\nu^-(Y_{r-1})=1$). Moreover, for every other inner fixed point component $Y$ we have $\nu^+(Y) >1$ (resp. $\nu^-(Y) >1$).
\end{lemma}

\begin{proof}
If $Y \subset Y_1$ is an irreducible component such that $\nu^+(Y)>1$, then there exists a positive dimensional family of invariant curves linking a point $y \in Y$ and the point $Y_0$. By bending-and-breaking, this family contains a non integral $1$-cycle $Z$, which is necessarily $\C^*$-invariant. 

By means of Lemma \ref{lem:AMFM}, the equalization of the action implies that $Z$ cannot be non reduced and irreducible. 
On the other hand, if it were reducible the intersection point of two components would be a fixed point with a weight in between $a_0$ and $a_1$, a contradiction.  

The proof is finished recalling \cite[Lemma 2.6 (2)]{WORS1}, which proves that there exists a unique irreducible fixed point component on which $\nu^+$ is equal to $1$. 
\end{proof}

\begin{remark}\label{rem:nonbordism}
Let us denote by $I(X,L)$ the set of indices $\{0,\dots,r\}$ satisfying the hypothesis $(\dagger)$ of Proposition \ref{prop:P1model}.
The above two lemmas imply that if $(X,L)$ is as in Setup \ref{ass:picZ}, then:
$$
I(X,L)=\left\{\begin{array}{ll}
\{0,\dots,r-1\}&\mbox{if }\dim(Y_0),\dim(Y_r)>0\\
\{1,\dots,r-1\}&\mbox{if }\dim(Y_0)=0,\dim(Y_r)>0\\
\{0,\dots,r-2\}&\mbox{if }\dim(Y_0)>0,\dim(Y_r)=0\\
\{1,\dots,r-2\}&\mbox{if }\dim(Y_0)=\dim(Y_r)=0\\
\end{array}\right.
$$
In particular, by Proposition \ref{prop:P1model},  for every $i\in I(X,L)$ we have a small modification $\widehat{X}(i,i+1)$ of $X^\flat$ that is a $\P^1$-bundle over $\GX(i,i+1)$.
\end{remark}

\subsection{The movable cone of $X^\flat$ and its stable base locus decomposition}\label{ssec:linsyst}
 
Recall that, given a $\Q$-divisor $D\in \Pic(Y)\otimes_{\Z}\Q$ on a normal variety $Y$, the {\em stable base locus of $D$} (see \cite[Definition~2.1.20]{L2}) is defined as:
$$
\B(D):=\bigcap_{\substack{m>0\\ mD\in\Pic(Y)}} \Bs(mD).
$$
A $\Q$-divisor $D\in \Pic(Y)\otimes_{\Z}\Q$  on a normal variety $Y$ is called {\em movable} if $\B(D)$ 
has codimension at least two. If $\HH^1(Y,\cO_Y)=0$, then being movable is a numerical property and 
we can define the {\em movable cone of } $Y$, denoted $\Mov(Y)$, as the convex cone in $\NU(Y)$ generated by the classes of movable divisors. Furthermore, we may decompose $\Mov(Y)$ in chambers on whose interior the stable base locus is constant; this is called the {\em stable base locus (SBL) decomposition} of the movable cone $Y$ (see, for instance, \cite[Section~4.1.3]{Hui}). 

Note that, for a variety $X$ in the situation of Setup \ref{ass:picZ}, as in Remark \ref{rem:RC}, we have that $\HH^1(X^{\flat},\cO_{X^{\flat}})=0$, hence it makes sense to talk about $\Mov(X^{\flat})$ and its SBL decomposition; this is the goal of this Section. 

Let us start by describing $\Mov(X^\flat)$. For simplicity we will focus in the case in which the  action is a bordism (i.e. $\dim(Y_i)>0$ for $i=0,r$), an we will discuss later how the same arguments apply in the non-bordism case.

\begin{proposition}\label{prop:movable1}
In the situation of Setup \ref{ass:picZ},  assuming that $\dim(Y_i)>0$, $i=0,r$, then the movable cone of $X^\flat$ is simplicial:
$$\Mov(X^\flat)=\cMov{X^\flat}=\langle L(0,\delta),L(0,0),L(\delta,\delta)\rangle.$$
\end{proposition}

\begin{proof}
The divisor $L(0,\delta)=\beta^*L$ is nef, hence movable. On the other hand, by Corollary \ref{cor:BSL}, $\Bs(mL(0,0))=\Bss^\flat_{-1,1}$, for $m\gg 0$. Each irreducible component of $\Bss^\flat_{-1,1}$ is either the source or has codimension in $X$ equal to $\nu^-(Y)$ for an inner fixed point component $Y$. By Remark \ref{rem:extremaldiv}, Lemma \ref{lem:bordism2}, and Lemma \ref{lem:bordism}, none of these components is a divisor and $L(0,0)$ is movable. A similar proof shows that $L(\delta,\delta)$ is movable, so we get an inclusion
$$\langle L(0,\delta),L(0,0),L(\delta,\delta)\rangle\subseteq\Mov(X^\flat)\subseteq\cMov{X^\flat}.$$ 
On the other hand, we consider the numerical classes of the following curves:
\begin{itemize}
\item the closure $C_{\gen}$ of the general $\C^*$-orbit in $X^\flat$;
\item a line $C_i$ in a fiber of the projective bundle $\beta:Y_i^\flat\to Y_i$, $i=0,r$. 
\end{itemize}

The loci of the deformations of these curves have codimension $0$ and $1$, respectively, hence their intersection number with every class in $\cMov{X^\flat}$ must be nonnegative. Since, by Lemma \ref{lem:AMFM}, we have $\beta^*L\cdot C_{\gen}=\delta$, and clearly: 
$$\beta^*L\cdot C_i=0,\,\,\,Y_i^\flat \cdot C_{\gen}=1,\,\,\,Y_i^\flat\cdot C_j=-\delta_{ij},$$
the above nonnegativity conditions can easily be rewritten as
$$\cMov{X^\flat}\subseteq \langle L(0,\delta),L(0,0),L(\delta,\delta)\rangle,$$
and the statement follows.  
\end{proof}

\begin{remark}\label{rem:movable2}
The above proof shows that the dual of $\Mov(X^\flat)$ in the case in which $\dim(Y_0),\dim(Y_r)>0$ is the cone generated by the classes $[C_{\gen}],[C_0],[C_r]$. If this condition is not fulfilled, we need to add some extra classes of curves (whose deformation loci have codimension $1$) to the set of generators of $\Mov(X^\flat)^\vee$. More precisely, denoting by $C_{1,r}$, $C_{0,r-1}$ two irreducible $\C^*$-invariant curve linking $Y_1$ with $Y_r$, and $Y_0$ with $Y_{r-1}$, respectively, an  argument similar to the one in the previous proof  gives:
\begin{itemize}
\item If $\dim(Y_0)=0$, $\dim(Y_r)>0$, then %\begin{multline*}
$$\Mov(X^\flat)=\langle [C_{\gen}],[C_0],[C_r], [C_{1,r}]\rangle^\vee=\langle L(0,a_1),L(a_1,a_1),L(\delta,\delta),L(0,\delta)\rangle.$$%\end{multline*}
\item If $\dim(Y_0)>0$, $\dim(Y_r)=0$, then 
\begin{multline*}\Mov(X^\flat)=\langle [C_{\gen}],[C_0],[C_r], [C_{0,r-1}]\rangle^\vee=\\=\langle L(0,0),L(a_{r-1},a_{r-1}),L(a_{r-1},\delta),L(0,\delta)\rangle.\end{multline*}
\item If $\dim(Y_0)=0$, $\dim(Y_r)=0$, then \begin{multline*}\Mov(X^\flat)=\langle [C_{\gen}],[C_0],[C_r], [C_{1,r}], [C_{0,r-1}]\rangle^\vee=\\=\langle L(0,a_1),L(a_1,a_1),L(a_{r-1},a_{r-1}),L(a_{r-1},\delta),L(0,\delta)\rangle.\end{multline*}
\end{itemize}
The positions of the generators of (an affine slice of) the movable cone in each case has been represented in Figure \ref{fig:movbordism}.
\end{remark}

The SBL decomposition of $\Mov(X^{\flat})$ can be now read out of the following statement, that is a consequence of Corollary \ref{cor:BSL}, 

\begin{corollary}\label{cor:SBL}
For every $\tau_-,\tau_+\in[0,\delta]\cap \Q$, $\tau_-\leq\tau_+$ we have:
$$\B(L(\tau_-,\tau_+))=\Bss^\flat_{i(\tau_-)j(\tau_+)}.
$$
\end{corollary}

In particular, the chambers of the SBL decomposition of $\Mov(X^{\flat})$ can be described 
as follows. For every set of indices $(i,j)$, $0\leq i\leq j\leq r$ we set:
$$
N_{i,j}:=\left\{mL(\tau_-,\tau_+)\in \NU(X^{\flat})
\left|\begin{array}{l} 0<\tau_-<\tau_+<\delta,\,\, m\geq 0\\ \tau_-\in(a_i,a_{i+1}),\,\,\tau_+\in(a_{j-1},a_j)\end{array} \right.\right\}.
$$
The previous statement tells us that $N_{i,j}$ is an SBL-chamber whenever it is contained in $\Mov(X^{\flat})$, and this is the case for every $(i,j)$, $i< j$, except in the following two cases (see Remark \ref{rem:movable2}):
\begin{itemize}
\item $Y_0$ is a point and $(i,j)=(0,1)$;
\item $Y_r$ is a point and $(i,j)=(r-1,r)$;
\end{itemize}
We have represented (an affine slice of) the SBL decomposition of the movable cone of $X^\flat$ in Figure \ref{fig:movbordism}.

\begin{center}
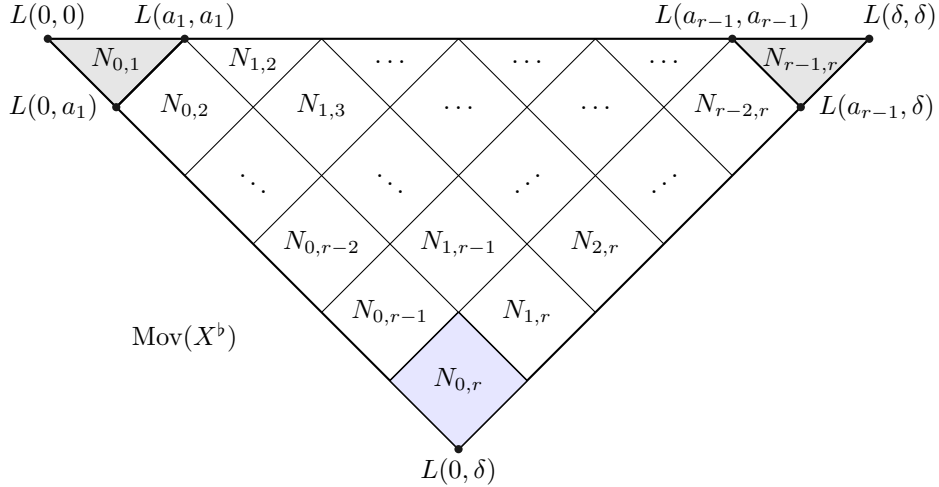
\begin{figure}[h!]
\begin{tikzpicture}[scale=0.91]
\draw[fill=gray!20]    (-6,0) -- (-4,0) -- (-5,-1) -- (-6,0);
\draw[fill=gray!20]    (6,0) -- (4,0) -- (5,-1) -- (6,0);
\draw [thick] (-6,0) -- (6,0) -- (0,-6) -- (-6,0);
\draw[thick] (-5,-1) -- (-4,0); 
\draw[thick](-5,-1) -- (-4,0);
\draw[thick] (-1,-5) -- (0,-4); 
 \draw[thick] (1,-5) -- (0,-4); 
\draw(-4,0) -- (1,-5);
\draw (-4,-2) -- (-2,0) -- (2,-4);
\draw (-3,-3) -- (0,0) -- (3,-3);
\draw (-2,-4) -- (2,0) -- (4,-2);
\draw (-1,-5) -- (4,0); 
\draw[thick] (4,0)--(5,-1);
\draw [fill=blue!10] (0,-6)--(1,-5)--(0,-4)--(-1,-5)--(0,-6);
\node [below] at (-4,-4) {$\Mov(X^\flat)$}; \node at (0,-5) {$N_{0,r}$};
\node at (-1,-4) {$N_{0,r-1}$};\node at (1,-4) {$N_{1,r}$};
\node at (-2,-3) {$N_{0,r-2}$};\node at (0,-3) {$N_{1,r-1}$};\node at (2,-3) {$N_{2,r}$};
\node at (-3,-2) {$\ddots$};\node at (-1,-2) {$\ddots$};\node at (1,-2) {$\iddots$};\node at (3,-2) {$\iddots$};
\node at (-4,-1) {$N_{0,2}$};\node at (-2,-1) {$N_{1,3}$};\node at (0,-1) {$\ldots$};\node at (2,-1) {$\ldots$};\node at (4,-1) {$N_{r-2,r}$};
\node [below] at (-5,0.00) {$N_{0,1}$};\node [below] at (-3,0.00) {$N_{1,2}$};\node [below] at (-1,-0.18) {$\ldots$};\node [below] at (1,-0.18) {$\ldots$};\node [below] at (3,-0.18) {$\ldots$} ;\node [below] at (5,0.00) {$N_{r-1,r}$};
\fill[black!90!white] (-6,0) circle (0.6mm); \node[anchor=south]  at (-6,0) {$L(0,0)$\qquad};
\fill[black!90!white] (-5,-1) circle (0.6mm); \node[anchor=east] at (-5,-1) {$L(0,a_1)\,\,$};
\fill[black!90!white] (-4,0) circle (0.6mm); \node[anchor=south] at (-4,0) {$L(a_1,a_1)$};
\fill[black!90!white] (4,0) circle (0.6mm); \node[anchor=south] at (4,0) {$L(a_{r-1},a_{r-1})$};
\fill[black!90!white] (5,-1) circle (0.6mm); \node[anchor=west] at (5,-1) {$\,\,L(a_{r-1},\delta)$};
\fill[black!90!white] (6,0) circle (0.6mm); \node[anchor=south] at (6,0) {$\qquad\,\,L(\delta,\delta)$};
\fill[black!90!white] (0,-6) circle (0.6mm); \node[anchor=north] at (0,-6) {$L(0,\delta)$};
\end{tikzpicture}
\caption{The movable cone of $X^\flat$ and its SBL decomposition.\label{fig:movbordism}} \end{figure}
\end{center}

\subsection{The Mori chamber decomposition of $X^\flat$}\label{ssec:SQM}

We will now describe the {\em Mori chamber decomposition} of the movable cone of $X^\flat$, which in general is only known to be a refinement of the stable base locus decomposition of the movable cone  (see for instance \cite[Remark~2.5]{LMR}). In our case we will show:

\begin{theorem}\label{thm:Mori}
In the situation of Setup \ref{ass:picZ}, the Mori chamber decomposition of $X^\flat$ equals its SBL decomposition, that is, the Mori chambers of $X^\flat$ are the sets $N_{i,j}\subset\Mov(X^\flat)$.
\end{theorem}

Note that if $\dim(Y_0)=0$ (resp. $\dim(Y_r)=0$) then $\Mov(X^\flat)$ does not contain the chamber $N_{0,1}$ (resp. $N_{r-1,r}$); in the case in which $X^\flat$ is a bordism we have a chamber for every pair of indices $(i,j)\in\{0,\dots,r\}^2$, $i<j$. Let us set
$$
\cM:=\left\{(i,j)|\,\, N_{i,j}\subset \Mov(X^\flat)\right\}.
$$
In order to prove Theorem  \ref{thm:Mori} we will use a recursive argument that allows us to prove the following more detailed version of it:

\begin{proposition}\label{prop:precise}
In the situation of Setup \ref{ass:picZ}, for every $(i,j)\in \cM$ there exists a smooth projective variety $X(i,j)$,  with a B-type equalized $\C^*$-action of criticality $j-i$, and a $\C^*$-equivariant small modification $\varphi_{i,j}:X^\flat\dashrightarrow X(i,j)$ such that:
\begin{itemize}
\item[(i)] the indeterminacy locus of $\varphi_{i,j}$ does not meet any inner fixed point component of $L$-weight $a_k$, for $i<k<j$;
\item[(ii)] $\varphi_{i,j}$ maps isomorphically the inner fixed point components of $X^\flat$ of weights $a_k$, $i<k<j$ to the inner fixed point components of $X(i,j)$, and preserves the values of $\nu^\pm$ on these components;
\item[(iii)] $\Nef(X(i,j))=\overline{N_{i,j}}$.
\end{itemize}
\end{proposition}
 
 \begin{proof}
We start by setting $X(0,r):=X^{\flat}$, $\varphi_{0,r}:=\id$.

By Corollary \ref{cor:SBL}, any $\Q$-divisor in $\overline{N_{0,r}}$ is semiample, hence nef, so we may conclude that $\overline{N_{0,r}}\subset \Nef(X(0,r))$. On the other hand, by the same Corollary, movable $\Q$-divisors outside of $\overline{N_{0,r}}$ are not semiample, hence not ample, so the equality follows.

Take $(i,j) \in \cM$ and assume that the statement holds for such a pair. We will show that if $(i+1,j)\in \cM$, then the statement is also true for $(i+1,j)$; an analogous argument yields that the result holds for $(i,j-1)$ when $(i,j-1)\in \cM$. 

First of all, we note that assuming that $(i+1,j)\in \cM$ implies that either 
\begin{itemize}
\item[(a)] $j-i\geq 3$, or
\item[(b)] $j-i=2$ and $i<r-2$, or
\item[(c)]   $i=r-2$, $j=r$,  and $\dim(Y_r)>0$.
\end{itemize}
This follows from the description of $\Mov(X^\flat)$ given in Proposition \ref{prop:movable1} and Remark \ref{rem:movable2}, and of its SBL decomposition given at the beginning of this Section. 

Now we use  property (ii) to identify the inner fixed point components of $X(i,j)$ with some inner fixed point components of $X^\flat$; we also choose an ample $\Q$-divisor $D:=(\varphi_{i,j})_* L(\tau_-,\tau_+)\in N_{i,j}$. By Lemma \ref{lem:AMFM}, the $D$-weights of the inner fixed point components of $X(i,j)$ are $a_{i+1}-\tau_-,\dots,a_{j-1}-\tau_-$. 

Since $\nu^{\pm}$ is preserved by $\varphi_{i,j}$ on these inner fixed point components, by Lemmas \ref{lem:bordism2}, \ref{lem:nonbordism}, it follows that, in each of the situations (a),(b),(c) described above, the pair $(X(i,j),D)$ satisfies the hypotheses of Theorem \ref{thm:flip}. We then consider the smooth $\C^*$-equivariant small modification $\psi^-:X(i,j)\dashrightarrow \cS X(i,j)^-$ provided by that statement and set:
$$
X(i+1,j):=\cS X(i,j)^-,\quad \varphi_{i+1,j}:=\psi^-\circ \varphi_{i,j}.
$$
Let us now show that $X(i+1,j)$ and $\varphi_{i+1,j}$ satisfy the requirements of the Proposition. Note that the induced $\C^*$-action on $X(i+1,j)$ is equalized of B-type by Theorem 3.1, so we are left to prove the properties (i), (ii), (iii).

By assumption, the indeterminacy locus of $\varphi_{i,j}$ does not meet any fixed point component of $L$-weight  $a_k$, for $i<k<j$. On the other hand, as stated in Theorem \ref{thm:flip}, the indeterminacy locus of $\psi^-$ does not meet any fixed point component of $X(i,j)$ of $D$-weight $a_{k}-\tau_-$, for $i+1<k<j$, that correspond via $\varphi_{i,j}$ to the fixed point components of $X^\flat$ of $L$-weight $a_{k}$, for $i+1<k<j$, so (i) follows. 

Property (ii) follows in a similar way: both $\varphi_{i,j}$ and $\psi^-$ map isomorphically the fixed point components of $D$-weights $a_{k}$, for  $i+1<k<j$.

For part (iii) we denote by $Y_0^-$ the sink of $X(i,j)$, and note that $Y_0^-=(\varphi_{i,j})_*(Y_0)$. Then, by Theorem \ref{thm:flip}, 
$$\psi^-_*(D-\tau Y_0^-)=(\varphi_{i+1,j})_*(L(\tau_-+\tau,\tau_+))$$ 
is ample on $X(i+1,j)$, for every $\tau\in (a_{i+1}-\tau_-,a_{i+2}-\tau_-)$. Since this holds for every $L(\tau_-,\tau_+)\in N_{i,j}$, setting $\tau':=\tau_-+\tau$, we conclude that $(\varphi_{i+1,j})_* L(\tau',\tau_+)$ is ample on $X(i+1,j)$ for every $L(\tau',\tau_+)\in N_{i+1,j}$. 

On the other hand, since the Mori chamber decomposition is a refinement of the SBL decomposition, the ample cone of $X(i+1,j)$ must be contained in $N_{i+1,j}$, and we conclude that $\Nef(X(i+1,j))=\overline{N_{i+1,j}}$.   
\end{proof}

We finish this section with some observations regarding the GIT quotients of the small modifications of $X^\flat$. In the following statement we consider the set of indices $I(X,L)$ introduced in Remark \ref{rem:nonbordism}.

\begin{corollary}\label{cor:GITSQM1}
In the Setup \ref{ass:picZ}, for every $i\in I(X,L)$: 
\begin{itemize}
\item[(i)] the closure of  $N_{i,i+1}$ is the nef cone of the $\P^1$-bundle $\widehat{X}(i,i+1)$;
\item[(ii)] the intersection
$$\overline{N_{i,i+1}}\cap C_{\gen}^\perp=\left\{mL(\tau,\tau)\in 
\cMov{X^\flat}\left|\begin{array}{l} m\geq 0\\a_i\leq\tau\leq a_{i+1}\end{array}\right.\right\}$$
is the nef cone of the geometric quotient $\GX(i,i+1)$;
\item[(iii)] the intersection $\cMov{X^\flat}\cap C_{\gen}^\perp$ equals $\cMov{\GX(i,i+1)}$.
\end{itemize} 
\end{corollary}

\begin{proof}
First of all, as noted in Remark \ref{rem:nonbordism}, for each such $i$ the nef cone of $\widehat{X}(i,i+1)$ is the closure of one of the Mori chambers of $X^\flat$, and the projection $\widehat{X}(i,i+1)\to \GX(i,i+1)$ corresponds to a facet of that chamber. On the other hand, by Proposition \ref{prop:semisection2} and Lemma \ref{lem:projections2} 
$$
\GX(i,i+1)=\Proj\left(\bigoplus_{\substack{m\geq 0\\m\tau\in\Z}}\HH^0(X,mL)_{m\tau}\right)=\Proj\left(\bigoplus_{\substack{m\geq 0\\m\tau\in\Z}}\HH^0(X^\flat,mL(\tau,\tau))\right),
$$
which implies that $L(\tau,\tau)$ is a supporting $\Q$-divisor of the rational contraction 
$X^\flat\dashrightarrow \widehat{X}(i,i+1)\to \GX(i,i+1)$. Since the only chamber containing $L(\tau,\tau)$ in its boundary is $N_{i,i+1}$, we conclude that $\Nef\big(\widehat{X}(i,i+1)\big)=\overline{N_{i,i+1}}$, that is, (i).

Moreover, $\Nef\big(\GX(i,i+1)\big)$ must be the facet of $\overline{N_{i,i+1}}$ containing $L(\tau,\tau)$, which is $\overline{N_{i,i+1}}\cap C_{\gen}^\perp$, and so (ii) follows.

In order to proof part (iii) we first note that the geometric quotients $\GX(i,i+1)$ $i\in I(X,L)$ are all isomorphic in codimension one by construction, so we may conclude by showing that the extremal rays of $\cMov{X^\flat}\cap C_{\gen}^\perp$ do not correspond to small contractions. For simplicity let us consider only the left hand side extremal ray $R$, and denote by $\GX(i_0,i_0+1)$ the geometric quotient whose nef cone contains that ray; according to Remark \ref{rem:nonbordism} we have two situations: either $i_0=0$, or $i_0=1$. 

If $i_0=0$, the contraction $\GX(0,1)\to \GX(0,0)$ associated to the ray $R$ is the projective bundle $Y_0^\flat\to Y_0$. If else $i_0=1$, the contraction $\GX(1,2)\to \GX(1,1)$ is divisorial: in fact by Lemma \ref{lem:nonbordism} in this case there is a unique component $Y_1$ of weight $a_1$, and $\nu^+(Y_1)=1$; therefore the subset of $\GX(1,2)$ parametrizing orbits converging to $Y_1$ at infinity is a divisor (by the Bia{\l}ynicki-Birula decomposition) that gets contracted onto $Y_1\subset \GX(1,1)$. This concludes the proof.
\end{proof}

\begin{remark}\label{rem:GITSQM1}
Similar arguments show that, for every $i<j$, the geometric (resp. semigeometric) quotients of the small modification $X(i,j)$ are the varieties $\GX(k,k+1)$, $i\leq k\leq j-1$ (resp. $\GX(k,k)$, $i\leq k\leq j$). In particular, the sink and the source of $X(i,j)$ are $\GX(i,i)$, $\GX(j,j)$. 
\end{remark}

\begin{remark}\label{bigcone}
In the case in which the sink of the action in $X$ is a point, we have seen that the simplicial cone generated by $L(0,0),L(0,a_1),L(a_1,a_1)$ is not contained in $\Mov(X^\flat)$. However, it is still contained in the pseudo-effective cone of $X^\flat$, and the $\Q$-divisors contained in it have a geometric interpretation. 

In fact, the complete linear system of a large enough integral multiple of a $\Q$-divisor in $\langle L(0,0),L(0,a_1),L(a_1,a_1)\rangle \setminus \langle L(0,a_1),L(a_1,a_1)\rangle$ has a fixed point component supported in the unique divisor $\overline{X^-(Y_1)}$. In particular, the divisors of the form $m L(\tau,\tau)$, $\tau\in\Q\cap(0,a_1)$, $m\gg 0$ still define rational maps from $X^\flat$ to the geometric quotient $\GX(0,1)=\GX(1,1)$, as we knew from GIT (see Proposition \ref{prop:semisection2}). 
\end{remark}

%!TEX root = WORS3.tex

\section{Examples: equalized actions on rational homogeneous manifolds}\label{sec:RH}

The paper \cite{WORS1}, in which the birational geometry of equalized $\C^*$-actions of small bandwidth was considered, shows that important examples of equalized actions can be found using  Representation Theory. Let us start this section by presenting some of these. As in Section \ref{sec:movbordism}, we will be interested in smooth projective varieties of Picard number one admitting an equalized nontrivial $\C^*$-action. We will also consider a linearization of the action on the ample generator $L$ of $\Pic(X)$ and denote the bandwidth and the criticality of the action by $\delta$ and $r$, respectively.

First of all, we will introduce the notation that we will use to describe our varieties and actions.

\begin{notation}\label{notn:RH}
Let $G$ be a semisimple algebraic group, with Lie algebra $\fg$. We consider  a Borel subgroup $B\subset G$ and a Cartan subgroup $H\subset B\subset G$, denote by $\Phi$ the root system of $G$ with respect to $H$, by $W=N_G(H)/H$ the Weyl group of $G$, 
by $\Delta=\{\alpha_1,\dots,\alpha_r\}$ the base of positive simple roots of $\Phi$ induced by $B\supset H$, by $\Phi^+$ the set of positive roots determined by $B$,
and by $\cD$ the Dynkin diagram of $G$. We will assume that $\cD$ is connected, i.e., that the Lie algebra $\fg$ of $G$ is simple (and we will say that $G$ is a {\it simple} algebraic group), and that $G$ is simply connected. In particular the lattice of characters $\Mo(H)$ of $H$ coincides with the lattice of weights of $\fg$, generated by the {\em fundamental dominant weights} $\{\omega_1,\dots,\omega_n\}$. 

A fundamental weight $\omega_i$ defines a representation $V_{\omega_i}$ such that the unique closed $G$-orbit in the projectivization $\P(V_{w_i})$ is a {\em rational homogeneous variety}. This variety is completely determined by the choice of a node $i$ in the Dynkin diagram $\cD$, hence we will denote it by $\cD(i)$.\end{notation}

\subsection{Examples with $r=1$}\label{ssec:exr1} The first example of variety admitting an equalized action of Picard number one and criticality $r=1$ is the projective space $\P^n$, together with the linear action that fixes the points of a hyperplane $H\subset \P^n$ and a point $P\notin H$. Other examples of this kind can be found within the class of horospherical varieties. Following \cite{Pas}, besides $\P^n$, smooth projective horospherical varieties of Picard number one are determined by the choice of some triples $(\cD,i,j)$, where $\cD$ is the Dynkin diagram of a simple Lie algebra and $i,j$ are two nodes of the diagram. One requires the morphisms $\cD(i,j)\to \cD(i)$, $\cD(i,j)\to \cD(j)$ to be projective bundles, denotes by $\cL_i,\cL_j$ the pullbacks of the ample generators of $\Pic(\cD(i))$, $\Pic(\cD(j))$, and constructs the associated horospherical variety as the contraction of the decomposable $\P^1$-bundle $\P(\cL_i\oplus \cL_j)\to X$ via the corresponding tautological line bundle $\cO(1)$ (see \cite[Section 4]{WORS1} for details). The variety in question inherits a $\C^*$-action induced by the natural one on $\P(\cL_i\oplus \cL_j)$, which has $\delta=r=1$; its (extremal) fixed point components  are, clearly, $\cD(i)$, $\cD(j)$. Table \ref{tab:bw1} contains the complete list of these varieties.

\begin{table}[h!!]
\caption{Smooth projective horospherical varieties of Picard number one.\label{tab:bw1}}
\begin{tabular}{|l|l|l|l|}
\hline
$\cD(i)$&$\cD(j)$&$n,i,j$&$X$\\\hline\hline
$\DA_n(1)$&$\DA_n(n)$&$n\geq 2$&$\DD_{n+1}(1)$\\\hline
$\DA_n(i)$&$\DA_n(i+1)$&$n\geq 3,\,\,i<n$&$\DA_{n+1}(i+1)$\\\hline
$\DB_n(n-1)$&$\DB_n(n)$&$n\geq 3$&not homogeneous\\\hline
$\DB_3(1)$&$\DB_3(3)$&&not homogeneous\\\hline
$\DC_n(i+1)$&$\DC_n(i)$&$n\geq 2,\,\,i<n$&not homogeneous\\\hline
$\DD_n(n-1)$&$\DD_n(n)$&$n\geq 4$&$\DB_n(n)$\\\hline
$\DF_4(2)$&$\DF_4(3)$&&not homogeneous\\\hline
$\DG_2(2)$&$\DG(1)$&&not homogeneous\\\hline
\end{tabular}
\end{table}

\begin{center}
\begin{figure}[h!]
\begin{tikzpicture}[scale=0.4]
\draw (-3,0) -- (3,0) -- (0,-3) -- (-3,0);
\fill[black!90!white] (-3,0) circle (0.6mm); \node[anchor=south] at (-3,0) {$L(0,0)$};
\fill[black!90!white] (3,0) circle (0.6mm); \node[anchor=south] at (3,0) {$L(1,1)$};
\fill[black!90!white] (0,-3) circle (0.6mm); \node[anchor=north] at (0,-3) {$\,L(0,1)$};
\end{tikzpicture}
\caption{The movable cone of $X^\flat$ in the case of criticality one.\label{fig:movable_r1}}
\end{figure}
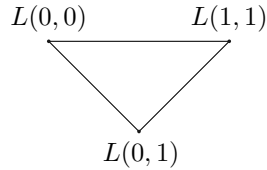
\end{center}
\subsection{Examples with $r=2$}\label{ssec:exr2}
By \cite[Theorem~4.1]{RW}, the only smooth projective varieties admitting an equalized action of bandwidth two with isolated extremal fixed points are the smooth quadrics, $\DB_n(1)$, $\DD_n(1)$. 

If $X$ is a rational homogeneous space of Picard number one and $L$ is the ample generator of $\Pic(X)$, examples in which $r=2$ can be found among adjoint varieties. Let us recall that, given a simple Lie algebra $\fg$, its {\em adjoint variety} is the unique closed orbit of the action of the adjoint group of $\fg$ on $\P(\fg^\vee)$. In \cite[Section~7]{WORS1} all the possible adjoint varieties admitting an equalized $\C^*$-action with $\delta=r=2$ have been described in terms of short gradings of Lie algebras  (see Table \ref{tab:bw2}). 
\begin{table}[!h!]
\caption{Adjoint varieties of Picard number one with an equalized action of bandwidth two.\label{tab:bw2}}
\begin{tabular}{|HcH|c|c|}
\hline
type&$X_{\ad}$&short gradings&$Y_0,Y_2$&$Y_{1}$\\
\hline\hline
$\DB_{m}$&$\DB_{m}(2)$&$1$&$\DB_{m-1}(1)$&$\DB_{m-1}(2)$\\\hline
$\DD_{m}$&$\DD_{m}(2)$&$1$&$\DD_{m-1}(1)$&$\DD_{m-1}(2)$\\\hline
$\DD_m$&$\DD_m(2)$&$m-1,m$&$\DA_{m-1}(2)$&$\DA_{m-1}(1,m-1)$\\\hline
$\DE_6$&$\DE_6(2)$&$1,6$&$\DD_5(5)$&$\DD_5(2)$\\\hline
$\DE_7$&$\DE_7(1)$&$7$&$\DE_6(1)$&$\DE_6(2)$\\\hline
\end{tabular}
\end{table}

We have represented the movable cones of these varieties in Figure \ref{fig:movable_r2}. Note that the arguments in \cite[Section~6]{WORS1} provide, for every smooth projective variety $X$ together with an equalized action with $\delta=r=2$ and non isolated extremal fixed points, two small modifications of $X^\flat$ that are $\P^1$-bundles. They correspond, in the language introduced here, to the varieties $X(0,1)$ and $X(1,2)$, whose nef cones are the two triangular chambers in the right hand side of Figure \ref{fig:movable_r2}.

\begin{center}
\begin{figure}[h!]
\begin{tikzpicture}[scale=0.5]
\draw (-3,2) -- (-1,0) -- (-3,-2) -- (-5,0) --  (-3,2);
\fill[black!90!white] (-3,2) circle (0.6mm); \node[anchor=south] at (-3,2) {$L(1,1)$};
 \fill[black!90!white] (-1,0) circle (0.6mm); \node[anchor=west] at (-1,0) {$L(1,2)$};
\fill[black!90!white] (-5,0) circle (0.6mm); \node[anchor=east] at (-5,0) {$L(0,1)$};
\fill[black!90!white] (-3,-2) circle (0.6mm); \node[anchor=north] at (-3,-2) {$\,L(0,2)$};
\node[anchor=north] at (-3,-3) {$\Mov(\mbox{Q}^\flat)$};

\draw (7,2) -- (5,0) -- (7,-2) -- (9,0) --  (7,2);
\draw (5,0) -- (3,2) -- (11,2) -- (9,0);
\draw [fill=black!10] (3,2)--(5,0)--(7,2)--(3,2);
\draw [fill=black!10] (7,2)--(9,0)--(11,2)--(7,2);
\fill[black!90!white] (7,2) circle (0.6mm); \node[anchor=south] at (7,2) {$L(1,1)$};
 \fill[black!90!white] (3,2) circle (0.6mm); \node[anchor=south] at (3,2) {$L(0,0)$};
\fill[black!90!white] (11,2) circle (0.6mm); \node[anchor=south] at (11,2) {$L(2,2)$};
\fill[black!90!white] (7,-2) circle (0.6mm); \node[anchor=north] at (7,-2) {$\,L(0,2)$};
\node[anchor=north] at (7,-3) {$\Mov(X^{\flat}_{\ad})$};
\end{tikzpicture}
\caption{The movable cone of $X^\flat$ in the case of quadrics and of adjoint varieties of criticality two.\label{fig:movable_r2}}
\end{figure}
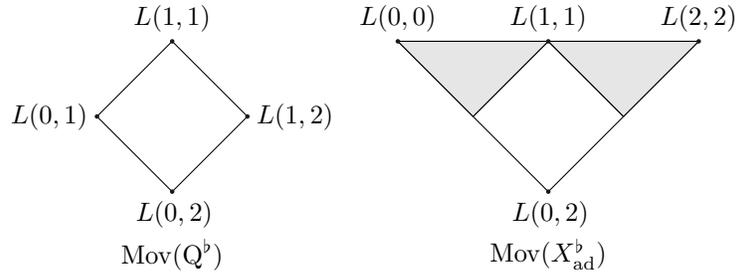
\end{center}

\subsection{Examples with $r=3$}\label{ssec:exr3} A complete list of smooth projective varieties admitting an equalized action of bandwidth three with isolated fixed points is known (see \cite{RW,WORS1}). In the case of Picard number one the varieties in question are all rational homogeneous and they can be described in terms of special Cremona transformations and Severi varieties (that appear as the inner fixed point components $Y_1,Y_2$ of the action). Their complete list is given in Table \ref{tab:bw3}.

\begin{table}[h!!]
\caption{Varieties of Picard number one with an equalized action of bandwidth three with isolated extremal fixed points.\label{tab:bw3}}
\begin{tabular}{|l|l|l|l|}
\hline
$X$&$Y_1,Y_2$\\\hline\hline
$\DC_3(3)$&$v_2(\DA_2(1))$\\\hline
$\DA_3(3)$&$\DA_2(1)\times \DA_2(2)$\\\hline
$\DD_6(6)$&$\DA_5(2)$\\\hline
$\DE_7(7)$&$\DE_6(1)$\\\hline
\end{tabular}
\end{table}

In \cite{WORS1} it has been shown that the blowup of such a variety along the sink and the source admits a small modification which is a decomposable $\P^1$-bundle (over the blowup of a projective space along a Severi variety). Again, the nef cone of this $\P^1$-bundle is the closure of the triangular chamber represented in Figure \ref{fig:movable_r3}, and the small modification from $X^\flat$ onto it is a composition of two (commuting) %Atiyah 
flips of the type described in Section \ref{sec:bord}.  

\begin{center}
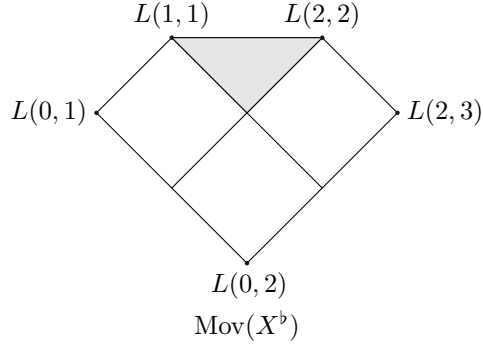
\begin{figure}[h!]
\begin{tikzpicture}[scale=0.5]
\draw (7,2) -- (5,0) -- (7,-2) -- (9,0) --  (7,2);
\draw  (5,0) -- (3,2) --(5,4) -- (7,2) -- (9,4) -- (11,2) -- (9,0);
\draw  (5,4) -- (9,4);
\fill[black!90!white] (5,4) circle (0.6mm); \node[anchor=south] at (5,4) {$L(1,1)$};
 \fill[black!90!white] (3,2) circle (0.6mm); \node[anchor=east] at (3,2) {$L(0,1)$};
\fill[black!90!white] (9,4) circle (0.6mm); \node[anchor=south] at (9,4) {$L(2,2)$};
 \fill[black!90!white] (11,2) circle (0.6mm); \node[anchor=west] at (11,2) {$L(2,3)$};
\fill[black!90!white] (7,-2) circle (0.6mm); \node[anchor=north] at (7,-2) {$\,L(0,2)$};
\draw [fill=black!10] (5,4)--(7,2)--(9,4)--(5,4);
\node[anchor=north] at (7,-3) {$\Mov(X^{\flat})$};
\end{tikzpicture}
\caption{The movable cone of $X^\flat$ in the case of variety of criticality three and isolated fixed points.\label{fig:movable_r3}}
\end{figure}
\end{center}

\subsection{Rational homogeneous examples with arbitrary criticality}\label{ssec:exarb} A description of $\C^*$-actions on rational homogeneous varieties of Picard number one  has been given in \cite{fra1}, with particular attention to the actions of minimal bandwidth. The author describes also when these actions are equalized. For the reader's convenience we include here some elementary facts that belong to that paper.

Let $i$ be a node of the connected Dynkin diagram $\cD$ of a group $G$ as in Notation \ref{notn:RH}, and set $X:=\cD(i)\subset\P(V_{\omega_i})$, $L:=\cO_{\P(V_{\omega_i})}(1)_{|X}$. In order to study $\C^*$-actions on $(X,L)$, we note first that we may always assume that the action  extends to the action of $H$ in $X$ by means of a homomorphism $\jmath:\C^*\to H$. The fixed points of the action of $H$ on $X$, as well as the weights of the action on $L$ and on the tangent spaces of $X$ at those points, can be written in terms of $\Phi$ and $W$, so that we may compute the $\C^*$-weights of $L$ and $T_X$ by means of the induced map $\jmath^*:\Mo(H)\to \Mo(\C^*)$. By choosing an isomorphism $\Mo(\C^*)\simeq \Z$, this map defines a grading of $\fg$, whose graded pieces are:
$$
\fg_0:=\fh\oplus\bigoplus_{\substack{\alpha\in\Phi\\\jmath^*(\alpha)=0}}\fg_{\alpha},\quad 
\fg_m:=\bigoplus_{\substack{\alpha\in\Phi\\\jmath^*(\alpha)=m}}\fg_{\alpha}, \,\,\,m\neq 0.
$$ 
One may then easily check that the relation among equalization and shortness of the grading found in \cite[Section~5]{WORS1} holds for every $X$ as above:

\begin{proposition}\label{prop}
The action of $\C^*$ on $X$ is equalized if the induced $\Z$-grading on $X$ is short, that is $\fg_m=0$ for $m\neq -1,0,1$. 
\end{proposition}

A classical example of this kind is the case of the Grassmannian of $i$-dimensional subspaces of an $(n+1)$-dimensional vector space $V$, which we denote by $X=\DA_n(i)$. A similar description can be done in the case of rational homogeneous varieties of type $\DB$, $\DC$ and $\DD$, and also in the exceptional cases, with the help of some representation-theoretical tools. We refer to \cite{fra1} for details. 

The line bundle $L$ is, in this case, the Pl\"ucker line bundle, that provides the embedding of $X$ into $\P(\wedge^iV)$. One may then construct equalized $\C^*$-actions on $X$ by choosing a decomposition 
$$
V=V_0\oplus V_1,
$$  
and a $\C^*$-action on $V$ whose weight is $0$ on $V_0$ and $1$ on $V_1$. Setting $k:=\dim V_0$, and assuming, without loss of generality, that $i\leq k\leq n-k+1$, one can then check that, for $m>0$:
$$
\fg_{m}=\bigoplus_{\substack{\alpha\in\Phi^+\\\alpha-m\alpha_k\in \Phi^+}}\fg_{\alpha}.
$$
In particular $\fg_m=0$ for $m>1$, that is, the grading is short.
One may compute that the fixed point components of the action are:
$$
\DA_{k-1}(i),\,\,\, \DA_{k-1}(i-1)\times \DA_{n-k}(1),\,\,\,\dots \,\,\,\DA_{k-1}(1)\times \DA_{n-k}(i-1),\,\,\, \DA_{n-k}(i),
$$
where we set $\DA_{r}(r+1)$ to be a point.

The $\C^*$-weights of $L$ on these components will be $0,1,\dots,i-1,i$, respectively. Note that our result provides a description of the Mori chamber decomposition of the movable cone of $\DA_n(i)^\flat$, which is not a Fano variety with the exception of the cases $i=1,n$. In particular, these examples show that one may find, within the class of rational homogeneous spaces, equalized $\C^*$-actions of arbitrary large criticality.

Note that will have isolated extremal fixed points in the following cases:
\begin{itemize}
\item $i=k<n+1-k$ (isolated sink),
\item $i=k=n+1-k$ (isolated sink and source).
\end{itemize}

In the latter case the sink and the source of $\DA_n(i)^\flat$ can be both identified with the projectivization of the space of $(n+1)/2 \times (n+1)/2$ matrices, and the birational map among them is known (we refer to \cite{MMW} for details) to be induced by the inversion function. This description allows us to decompose the (projectivized) inversion function as a sequence of a blowup, $(n+1)/2-2$ flips, and a blowdown.

\bibliographystyle{plain}
\bibliography{bibliomin}

\begin{thebibliography}{10}

\bibitem{AKMW}
Dan Abramovich, Kalle Karu, Kenji Matsuki, and Jaros{\l}aw W{\l}odarczyk.
\newblock Torification and factorization of birational maps.
\newblock {\em J. Amer. Math. Soc.}, 15(3):531--572, 2002.

\bibitem{ADHL}
Ivan {Arzhantsev}, Ulrich {Derenthal}, J\"urgen {Hausen}, and Antonio {Laface}.
\newblock {\em {Cox rings}}, volume 144.
\newblock Cambridge: Cambridge University Press, 2015.

\bibitem{BS}
Mauro~C. Beltrametti and Andrew~J. Sommese.
\newblock {\em The adjunction theory of complex projective varieties},
  volume~16 of {\em De Gruyter Expositions in Mathematics}.
\newblock Walter de Gruyter \& Co., Berlin, 1995.

\bibitem{BBC}
A.~Bia\l~ynicki Birula, J.~B. Carrell, and W.~M. McGovern.
\newblock {\em Algebraic quotients. {T}orus actions and cohomology. {T}he
  adjoint representation and the adjoint action}, volume 131 of {\em
  Encyclopaedia of Mathematical Sciences}.
\newblock Springer-Verlag, Berlin, 2002.

\bibitem{BB}
Andrzej Bia{\l}ynicki-Birula.
\newblock Some theorems on actions of algebraic groups.
\newblock {\em Ann. of Math. (2)}, 98:480--497, 1973.

\bibitem{BBS}
Andrzej Bia{\l}ynicki-Birula and Joanna \'Swi{\c e}cicka.
\newblock Complete quotients by algebraic torus actions.
\newblock In {\em Group actions and vector fields ({V}ancouver, {B}.{C}.,
  1981)}, volume 956 of {\em Lecture Notes in Math.}, pages 10--22. Springer,
  Berlin, 1982.

\bibitem{BBS1}
Andrzej Bia{\l}ynicki-Birula and Joanna \'{S}wi{\c{e}}cicka.
\newblock Generalized moment functions and orbit spaces.
\newblock {\em Amer. J. Math.}, 109(2):229--238, 1987.

\bibitem{BWW}
Jaros{\l}aw Buczy{\'n}ski, Jaros{\l}aw~A. Wi{\'s}niewski, and Andrzej Weber.
\newblock Algebraic torus actions on contact manifolds.
\newblock {\em To appear in J. Differ. Geom. Preprint ArXiv:1802.05002}, 2018.

\bibitem{CARRELL}
James~B. Carrell.
\newblock Torus actions and cohomology.
\newblock In {\em Algebraic quotients. {T}orus actions and cohomology. {T}he
  adjoint representation and the adjoint action}, volume 131 of {\em
  Encyclopaedia Math. Sci.}, pages 83--158. Springer, Berlin, 2002.

\bibitem{CS79}
James~B. Carrell and Andrew~J. Sommese.
\newblock Some topological aspects of $\mathbb{C}^*$-actions on compact kaehler
  manifolds.
\newblock {\em Commentarii mathematici Helvetici}, 54:567--582, 1979.

\bibitem{fra1}
Alberto Franceschini.
\newblock Minimal bandwidth $\mathbb{C}^*$-actions on generalized
  {G}rassmannians.
\newblock Preprint ArXiv:{\tt 2012.00498}, 2020.

\bibitem{HK}
Friedrich Hirzebruch and Kunihiko Kodaira.
\newblock On the complex projective spaces.
\newblock {\em J. Math. Pures Appl. (9)}, 36:201--216, 1957.

\bibitem{HuKeel}
Yi~Hu and Sean Keel.
\newblock Mori dream spaces and {GIT}.
\newblock {\em Michigan Math. J.}, 48:331--348, 2000.
\newblock Dedicated to William Fulton on the occasion of his 60th birthday.

\bibitem{Hui}
Jack Huizenga.
\newblock Birational geometry of moduli spaces of sheaves and {B}ridgeland
  stability.
\newblock In {\em Surveys on recent developments in algebraic geometry},
  volume~95 of {\em Proc. Sympos. Pure Math.}, pages 101--148. Amer. Math.
  Soc., Providence, RI, 2017.

\bibitem{IVERSEN}
Birger Iversen.
\newblock A fixed point formula for action of tori on algebraic varieties.
\newblock {\em Inv. Math.}, 16(3):229--236, 1972.

\bibitem{KKLV}
Friedrich Knop, Hanspeter Kraft, Domingo Luna, and Thierry Vust.
\newblock {\em Local Properties of Algebraic Group Actions}, pages 63--75.
\newblock Birkh{\"a}user Basel, Basel, 1989.

\bibitem{LMR}
Antonio Laface, Alex Massarenti, and Rick Rischter.
\newblock On {M}ori chamber and stable base locus decompositions.
\newblock {\em Trans. Amer. Math. Soc.}, 373(3):1667--1700, 2020.

\bibitem{L2}
Robert Lazarsfeld.
\newblock {\em Positivity in algebraic geometry. {I}}, volume~48 of {\em
  Ergebnisse der Mathematik und ihrer Grenzgebiete. 3. Folge. A Series of
  Modern Surveys in Mathematics [Results in Mathematics and Related Areas. 3rd
  Series. A Series of Modern Surveys in Mathematics]}.
\newblock Springer-Verlag, Berlin, 2004.

\bibitem{MMW}
Mateusz Micha{\l}ek, Leonid Monin, and Jaros{\l}aw~A. Wi\'sniewski.
\newblock Maximum likelihood degree and space of orbits of a
  $\mathbb{C}^*$-action.
\newblock {\em SIAM J. Appl. Algebra Geometry}, 1(5):60--85, 2021.

\bibitem{Morelli}
Robert {Morelli}.
\newblock {The birational geometry of toric varieties.}
\newblock {\em {J. Algebr. Geom.}}, 5(4):751--782, 1996.

\bibitem{MFK}
David Mumford, John Fogarty, and France Kirwan.
\newblock {\em Geometric invariant theory}, volume~34 of {\em Ergebnisse der
  Mathematik und ihrer Grenzgebiete (2) [Results in Mathematics and Related
  Areas (2)]}.
\newblock Springer-Verlag, Berlin, third edition, 1994.

\bibitem{WORS1}
Gianluca Occhetta, Eleonora~A. Romano, Luis E.~Sol\'a Conde, and Jaros{\l}aw~A.
  Wi\'sniewski.
\newblock Small bandwidth $\mathbb{C}^*$-actions and birational geometry.
\newblock Preprint ArXiv:1911.12129, 2019.

\bibitem{WORS2}
Gianluca Occhetta, Luis~E. Sol{\'a}~Conde, Eleonora~A. Romano, and Jaros\l
  aw~A. Wi\'sniewski.
\newblock High rank torus actions on contact manifolds.
\newblock {\em Selecta Math.}, 27(10), 2021.

\bibitem{OSS}
Christian Okonek, Michael Schneider, and Heinz Spindler.
\newblock {\em Vector bundles on complex projective spaces}.
\newblock Progress in Mathematics, 3. Birkh\"auser, Boston, Mass., 1980.

\bibitem{Pas}
Boris Pasquier.
\newblock On some smooth projective two-orbit varieties with {P}icard number 1.
\newblock {\em Math. Ann.}, 344(4):963--987, 2009.

\bibitem{ReidFlip}
Miles Reid.
\newblock What is a flip? %\url
  {http://homepages.warwick.ac.uk/staff/Miles.Reid/3folds}, 1992.

\bibitem{RW}
Eleonora~A. Romano and Jaros{\l}aw~A. Wi{\'s}niewski.
\newblock Adjunction for varieties with a $\mathbb{C}^*$ action.
\newblock {\em Transf. Groups}, 2020.
\newblock https://doi.org/10.1007/s00031-020-09627-8.

\bibitem{Thaddeus1996}
Michael Thaddeus.
\newblock Geometric invariant theory and flips.
\newblock {\em J. Amer. Math. Soc.}, 9(3):691--723, 1996.

\bibitem{Wlodarczyk}
Jaros{\l}aw {W{\l}odarczyk}.
\newblock {Birational cobordisms and factorization of birational maps.}
\newblock {\em {J. Algebr. Geom.}}, 9(3):425--449, 2000.

\bibitem{Wlodarczyk2009}
Jaros{\l}aw W{\l}odarczyk.
\newblock Simple constructive weak factorization.
\newblock In {\em Algebraic geometry---{S}eattle 2005. {P}art 2}, volume~80 of
  {\em Proc. Sympos. Pure Math.}, pages 957--1004. Amer. Math. Soc.,
  Providence, RI, 2009.

\end{thebibliography}
\end{document}